\documentclass[11pt]{article}
\usepackage{thm-restate}
\usepackage[T1]{fontenc}
\usepackage{tocbibind}
\usepackage{amsmath,amsthm,amssymb}
\usepackage{fancyhdr}
\usepackage[british]{babel}
\usepackage{geometry,mathtools}
\usepackage{enumitem}
\usepackage{algpseudocode}
\usepackage{subcaption}
\usepackage{dsfont}
\usepackage{leftidx}
\usepackage{centernot}
\usepackage{xstring}
\usepackage{colortbl}
 \usepackage{titlefoot}
\usepackage[section]{algorithm}
\usepackage{graphicx}
\usepackage[font={footnotesize}]{caption}
\usepackage[usenames,dvipsnames,table]{xcolor}
\usepackage{pstricks,tikz}
\usepackage[h]{esvect}
\usepackage{enumitem}

\global\long\def\Gnp{G(n,p)}

\usepackage[
bookmarksopen=true,
bookmarksopenlevel=1,
colorlinks=true,
linkcolor=darkblue,
linktoc=page,
citecolor=darkblue, urlcolor  = violet
]{hyperref}
\usepackage{bbm}

\tikzstyle{vertex}=[circle,draw=black,fill=black,inner sep=0,minimum size=0.2cm,text=white,font=\footnotesize]
\usepackage{verbatim}
\usetikzlibrary{calc}

\numberwithin{equation}{section}

\definecolor{darkblue}{rgb}{0,0,0.5}

\newdimen\margin
\def\textno#1&#2\par{
	\margin=\hsize
	\advance\margin by -4\parindent
	\setbox1=\hbox{\sl#1}
	\ifdim\wd1 < \margin
	$$\box1\eqno#2$$
	\else
	\bigbreak
	\hbox to \hsize{\indent$\vcenter{\advance\hsize by -3\parindent
			\it\noindent#1}\hfil#2$}
	\bigbreak
	\fi}

\newtheorem{theorem}{Theorem}[section]
\newtheorem{prop}[theorem]{Proposition}
\newtheorem{claim}[]{Claim}
\newtheorem{lemma}[theorem]{Lemma}
\newtheorem{cor}[theorem]{Corollary}

\theoremstyle{definition}

\newtheorem{defin}[theorem]{Definition}

\newcounter{stepenv}
\newenvironment{stepenv}[1][]{\refstepcounter{stepenv}}{}

\newcounter{step}[stepenv]

\newcounter{substep}[step]
\renewcommand{\thesubstep}{\thestep.\arabic{substep}}

\newcommand{\cB}{\mathcal{B}}

\newcommand{\cD}{\mathcal{D}}

\newcommand{\cG}{\mathcal{G}}
\newcommand{\cH}{\mathcal{H}}

\newcommand{\cK}{\mathcal{K}}

\newcommand{\cR}{\mathcal{R}}
\newcommand{\cS}{\mathcal{S}}
\newcommand{\cT}{\mathcal{T}}

\newcommand{\cV}{\mathcal{V}}
\newcommand{\cW}{\mathcal{W}}

\def\COMMENT#1{}
\def\TASK#1{}
\let\TASK=\footnote            

\begin{document}

	\title{Size-Ramsey numbers of graphs with maximum degree three}

	\author{
		Nemanja Dragani\'c \thanks{Mathematical Institute, University of Oxford, UK.
			Email: \href{mailto:nemanja.draganic@math.ethz.ch}{\nolinkurl{nemanja.draganic@maths.ox.ac.uk}}. Research supported by SNSF project 217926. Part of this research was conducted while N.D. was at ETH Z\"urich, Switzerland, and partially supported by SNSF grant 200021\_196965.
		}		
		\and
		Kalina Petrova \thanks{Institute of Science and Technology Austria (ISTA), Klosterneurburg 3400, Austria. Email: \href{mailto:kalina.petrova@ist.ac.at}{\nolinkurl{kalina.petrova@ist.ac.at}}.
    	This project has received funding from the European Union’s Horizon 2020 research and innovation programme under the Marie Skłodowska-Curie grant agreement No 101034413\includegraphics[width=5.5mm, height=4mm]{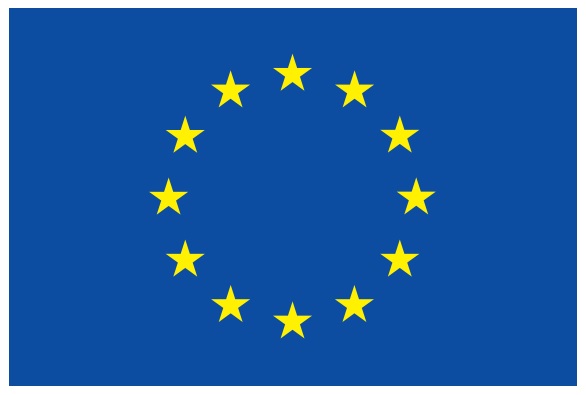}. Part of this research was conducted while K.P. was at the Department of Computer Science, ETH Z\"urich, Switzerland, supported by SNSF grant CRSII5 173721.}}
        
	\date{}
	
	\maketitle
	\begin{abstract}
	The size-Ramsey number $\hat{r}(H)$ of a graph $H$ is the smallest number of edges a (host) graph $G$ can have, such that for any red/blue colouring of $G$, there is a monochromatic copy of $H$ in $G$. Recently, Conlon, Nenadov and Truji\'c showed that if $H$ is a graph on $n$ vertices and maximum degree three, then $\hat{r}(H) = O(n^{8/5})$, improving upon the upper bound of $n^{5/3 + o(1)}$ by Kohayakawa, R\"odl, Schacht and Szemer\'edi. In this paper we show that $\hat{r}(H)\leq n^{3/2+o(1)}$. While the previously used host graphs were vanilla binomial random graphs, we prove our result using a novel host graph construction. Our bound hits a natural barrier of the existing methods.
	\end{abstract}
	\section{Introduction}
	\label{sec:introduction}
	Almost a century ago, Ramsey~\cite{ramsey1929on} showed a result which gave rise to one of the most important notions in combinatorics. His theorem, which was followed by extensive research, states that for every two integers $k$ and $\ell$, there exists $r = r(k,\ell)$ referred to as their \emph{Ramsey number}, which is the smallest integer such that in any colouring of the edges of $K_r$ in red and blue, there is either a red $K_k$ or a blue $K_{\ell}$. Determining the value $r(k,\ell)$ for general $k$ and $\ell$ has turned out to be difficult, and despite decades of research, there is still an exponential gap between the best-known lower and upper bounds~\cite{conlon2009a,erdos1947some,erdos1935a,sah2020diagonal,spencer1975ramsey,spencer1978asymptotic}. A natural generalization of this concept is the Ramsey number $r(H_1, H_2)$ of two graphs $H_1$ and $H_2$, which is the minimum $r$ such that any 2-colouring of the edges of $K_r$ contains a red $H_1$ or a blue $H_2$; we also write $r(k)$ for $r(k,k)$ and $r(H)$ for $r(H, H)$.
	
	For instance, a classic result of Gerencs\'er and Gy\'arf\'as~\cite{gerencser1967ramsey} shows that the Ramsey number of a path $P_n$ on $n$ edges satisfies $r(P_n)=m$ for $m=\lceil \frac{3n+1}{2}\rceil$. That is, however we colour the edges of $K_m$, there is a monochromatic copy of $P_n$ in it. Notice that the coloured graph $K_m$ has quadratically many edges in $n$, while $P_n$ only has linearly many edges. Is there a graph with much fewer edges than $K_m$ such that any 2-colouring of its edges again gives a monochromatic copy of $P_n$? Already in 1983, answering a \$100 question of Erd\H{o}s, Beck~\cite{beck1983on} showed that there is such a graph with only linearly many edges, which is evidently best possible. 
	
	To give a general framework for questions of this type, Erd\H{o}s, Faudree, Rousseau, and Schelp introduced the notion of \emph{size-Ramsey numbers}~\cite{erdos1978size}. Namely, given a graph $H$, the size-Ramsey number $\hat{r}(H)$ is the minimum number of edges a graph $G$ can have, such that $G$ is \emph{Ramsey for} $H$, that is, any $2$-colouring of $G$ contains a monochromatic copy of $H$. We refer to $G$ as the \emph{host graph} for $H$. 
	
	The concept of size-Ramsey numbers allows us to study the minimality of the host graph more precisely. It is always possible to take a large enough complete graph as the host graph, hence $\binom{r(H)}{2}$ is a trivial upper bound for $\hat{r}(H)$. This is also tight when $H$ is complete~\cite{erdos1978size}, but for other graphs $H$ the optimal host graph is often much sparser. Indeed, for certain graph classes, one can even show that the size-Ramsey number is linear in the number of vertices of $H$, a significant improvement over the trivial upper bound, which is always at least quadratic. Namely, in addition to the aforementioned result by Beck~\cite{beck1983on} that $\hat{r}(P_n) = O(n)$, in their very elegant paper Friedman and Pippenger~\cite{friedman1987expanding} proved that for every tree $T$ of bounded degree on $n$ vertices, $\hat{r}(T) = O(n)$. Furthermore, Haxell, Kohayakawa and {\L}uczak~\cite{haxell1995the} established that for the cycle on $n$ vertices, it holds that $\hat{r}(C_n) = O(n)$. Moreover, a linear upper bound of the size-Ramsey number was recently proved for long subdivisions of bounded degree graphs~\cite{draganic2022rolling} and for bounded degree graphs with bounded treewidth~\cite{berger2019size,kamcev2021size}. For further recent results on size-Ramsey numbers of (hyper)graphs, see~\cite{clemens2019size,clemens2021size,conlon2021three,conlon2022grids,draganic2021size,han2020multicolour,han2021size,letzter2021size}.
	
	Considering the mentioned results, one may suspect that the size-Ramsey number of every bounded degree graph is linear in its number of vertices. In fact, Beck~\cite{beck1983on} asked this question before most of these `positive examples' were discovered, but the answer was given much later and, perhaps surprisingly, was negative. Indeed, in 2000, R\"odl and Szemer\'edi~\cite{rodl2000size} showed that for every $n$, there are $n$-vertex graphs $H$ of maximum degree $3$ with $\hat{r}(H) \geq cn(\log{n})^{\frac{1}{60}}$ for some constant $c$. In the same paper, it was conjectured that this can be improved to $n^{1+\varepsilon}$ for some constant $\varepsilon>0$. Until very recently, $cn(\log{n})^{\frac{1}{60}}$ was still the best known lower bound on the size-Ramsey number of bounded degree graphs. Since the first version of this article, Tikhomirov~\cite{tikhomirov2024bounded} showed that there are $n$-vertex graphs $H$ of maximum degree three with size-Ramsey number $\hat{r}(H) \geq cn \exp(c\sqrt{\log{n}})$ for a universal constant $c$. However, R\"odl and Szemer\'edi's conjectured $n^{1 + \varepsilon}$ still remains out of sight.
	
	On the other hand, there have been some more recent important developments on the upper bound side. The baseline to be improved upon here is given by a classic result by Chvat\'al, R\"odl, Szemer\'edi, and Trotter~\cite{chvatal1983ramsey}, which shows that the Ramsey number of bounded degree graphs is linear in their number of vertices. This in turn gives a trivial quadratic upper bound for their size-Ramsey number.
    In 2011, Kohayakawa, R\"odl, Schacht and Szemer\'edi~\cite{kohayakawa2011sparse} were able to show that every $n$-vertex graph $H$ with maximum degree $\Delta$ satisfies $\hat{r}(H) \leq n^{2 - 1/\Delta + o(1)}$, thus bounding it away from quadratic. In the special case of $H$ being triangle-free and $\Delta \geq 5$, this result was improved to $n^{2-\frac{1}{\Delta-1/2}+o(1)}$ by Nenadov~\cite{nenadov2016ramsey}.
    
    Turning to particular instances of $\Delta$, note that for graphs $H$ of maximum degree $2$, the size-Ramsey number is linear. Indeed, such graphs have bounded treewidth, and so by~\cite{kamcev2021size}, we have that $\hat{r}(H) \leq O(n)$.
    Complementing the lower bound of R\"odl and Szemer\'edi for size-Ramsey numbers of cubic graphs $H$, and giving the first improvement over the general upper bound by Kohayakawa, R\"odl, Schacht and Szemer\'edi, recently Conlon, Nenadov and Truji\'c~\cite{conlon2022size} showed that $\hat{r}(H) = O(n^{8/5})$ for all cubic graphs $H$. With the additional assumption that $H$ is triangle-free, they further improved this bound to $\hat{r}(H) = O(n^{11/7})$, whereas when $H$ is bipartite, they proved that $\hat{r}(H) = O(n^{14/9})$.
	
    In this paper, we show that the size-Ramsey number of cubic graphs is at most $n^{3/2 + o(1)}$.
    As we will discuss below, this bound hits a natural barrier of the existing methods underlying previous work in this direction.
	\begin{theorem}
	\label{thm:main}
	The size-Ramsey number $\hat{r}(H)$ of every $n$-vertex graph $H$ with maximum degree $3$ satisfies $\hat{r}(H) \leq n^{3/2+o(1)}$.
	\end{theorem}
	
	In general, when it comes to size-Ramsey numbers, a natural candidate for a host graph is the binomial random graph $G(N,p)$\footnote{The binomial random graph $G(N,p)$ on $N$ vertices is obtained by adding each potential edge independently at random with probability $p$.}. Indeed, most of the size-Ramsey number upper bounds so far are achieved with $G(N,p)$ as a host graph. Typically, one proves that a graph sampled from $G(N,p)$ is with high probability\footnote{A property is said to hold \emph{with high probability} (w.h.p.) if it holds with probability tending to $1$ as $N\rightarrow \infty$.} Ramsey for $H$, which gives an upper bound of $O(N^2p)$ on the size-Ramsey number of $H$, since $G(N,p)$ w.h.p.\ has $O(N^2p)$ edges.
	
	However, there are some limits to what can be done with $G(N,p)$ as a host graph, coming from the fact that it is typically locally sparse. In particular, the upper bound of $O(n^{8/5})$ for cubic graphs $H$ achieved in~\cite{conlon2022size} is the best one can hope for using a vanilla binomial random graph as a host graph. For $p \ll N^{-2/5}$, the graph $G(N,p)$ is w.h.p.\ not even Ramsey for $K_4$~\cite{rodl1993lower,rodl1995threshold}, which can be a subgraph of $H$. We overcome this barrier by using a different host graph and additional new ideas for embedding $H$ into a monochromatic subgraph of the host graph.
	
    Our result pushes the known tools to their limit (up to the $o(1)$ term), most notably due to regularity inheritance, which is a property also used in the previous upper bounds on size-Ramsey numbers of bounded degree graphs. The employed approach which exploits this property breaks down when the number of edges in the host graph is asymptotically smaller than $n^{3/2}$, at least if they are `uniformly' distributed. To move past this limitation, it seems that entirely new ideas are required.
    
    We also note here that our proof yields a \emph{universality-type} result, meaning that for every red/blue colouring of our host graph $G$, there is a monochromatic subgraph of $G$ which contains all cubic graphs $H$ on $n$ vertices. In fact, as observed in \cite{alon2000universality}, any graph which contains all cubic graphs must have at least $\Omega(n^{4/3})$ edges (even without the colouring condition). Following \cite{kohayakawa2011sparse}, we say that a graph $G$ is \emph{partition universal} for a class of graphs $\mathcal{F}$ if for every $2$-colouring of the edges of $G$, there exists a monochromatic subgraph of $G$ which contains a copy of every graph in $\mathcal{F}$. Hence, the proof of our main result shows that an optimal partition universal graph for all $n$-vertex graphs with maximum degree three has at most $n^{3/2+o(1)}$ edges, complementing the aforementioned lower bound of $\Omega(n^{4/3})$.
	
	The rest of the paper is structured as follows. In Section~\ref{sec:proofoutline}, we give an overview of our approach. Section~\ref{sec:preliminaries} provides some technical tools which we later use. In Section~\ref{sec:decomposition}, we show a decomposition result for the cubic graph $H$, which prescribes the embedding process. Section~\ref{sec:host} is where we give the construction of our host graph and prove a number of useful properties of it. These set the stage for the proof of Theorem~\ref{thm:main} in Section~\ref{sec:proof}. Finally, in Section~\ref{sec:concluding}, we make some concluding remarks.
	
	\section{Proof outline}
	\label{sec:proofoutline}
	
	In this section, we present the main ideas of our approach. We first summarise the methods of Conlon, Nenadov and Truji\'c~\cite{conlon2022size} that give an upper bound of $O(n^{8/5})$, which we use as a starting point.
	
	Given a cubic graph $H$ on $cn$ vertices for some constant $c$, and an arbitrary 2-colouring of $G \sim G(n,p)$ with $p = \Omega(n^{-2/5})$, they use the regularity method to find 20 linear-sized sets of vertices in $G$, all pairs of which are regular in say blue, with some minimum density of order $p$. Next, to find a monochromatic copy of $H$ in $G$, they decompose $H$ into a number of vertex-disjoint parts $B_1, \dots, B_t$, and then embed the parts one by one in the blue subgraph induced by these 20 sets. Each part is either an induced path or an induced cycle of length at least $4$, and the decomposition is \emph{1-degenerate}---that is, each vertex $v \in B_i$ can have at most one neighbour $u_v$ in $B_1 \cup \dots \cup B_{i-1}$. 
	Hence, when they want to embed $B_i$, the `candidate set' for $v$ is the blue neighbourhood of $u_v$. The candidate sets can, with some care, be guaranteed to be of size $\Omega(np)$, and since $p=\Omega(n^{-1/2}\log n)$ one can ensure that each pair of candidate sets inherits regularity (as stated in Lemma~\ref{lem:typical_vertices}, which we borrow from \cite{conlon2022size}).
	
	There are two reasons why this approach reaches its limit when the host graph has $\Theta(n^{8/5})$ edges. The first one, as already mentioned in the introduction, is that for $p \ll n^{-2/5}$, the graph $G(n,p)$ is w.h.p.\ not Ramsey for $K_4$~\cite{rodl1993lower,rodl1995threshold} and so is not a suitable host graph. The second bottleneck is in the embedding of the $B_i$'s which are induced cycles. If $p = o(n^{-2/5})$, a copy of $C_4$ can no longer be embedded into the candidate sets as desired. Indeed, the technique used to embed the copy of $C_4$ relies on the K{\L}R conjecture (which is a theorem by now~\cite{balogh2015independent,conlon2014klr,nenadov2022new,saxton2015hypergraph}), and breaks down at the mentioned threshold for $p$.

	More generally, any decomposition which contains short cycles as parts is an obstacle for constructing sparser host graphs. In particular, if the length of the shortest induced cycle in the decomposition is $L$, then with the technique at hand one needs at least $n^{\frac{3}{2}+\frac{1}{4L-6}}$ edges in the host graph. Thus, to overcome this barrier and make the host graph sparser, a new decomposition of $H$ is needed---one that does not have short induced cycles as parts.
	
	We deal with these two hurdles by using a different host graph and a different decomposition of $H$, as well as introducing some other techniques.
	The first ingredient we need is a random graph model which is locally dense but globally sparse, thereby being ideal for dealing with the aforementioned issue with copies of $K_4$. We make use of the fact that for every constant $C$,  due to the existence of designs and subject to some divisibility conditions between $n$ and $C$, the edges of $K_n$ can be partitioned into copies of $K_C$~\cite{wilson1975existence}. Choosing a large enough constant $C$, we pick one such partition and then define the random graph model $G^C(n,p)$, in which the edges of each copy of $K_C$ in the partition are taken to be present independently from all other copies with probability $p$ (Definition~\ref{defin:gcnp}). Provided that $C \geq r(4,4)$, typically, every 2-colouring of $G^C(n,p)$ will have many monochromatic copies of $K_4$. This random graph model is an important building block of our host graph, but is not all there is to it, since our new decomposition requires a more involved host graph construction.
	
	To address the second obstacle mentioned above, we need a new decomposition of $H$ which avoids short induced cycles. To construct a host graph on $n^{3/2+o(1)}$ edges, one has to be able to avoid all cycles of length less than $L$, for any arbitrary constant $L$, as discussed above. Hence, in Section~\ref{sec:decomposition} we show Lemma~\ref{lem:decomposition} --- a decomposition result for cubic graphs similar in spirit to the one in \cite{conlon2022size}, yet fundamentally different from it. One part in our decomposition is a graph with bounded treewidth, and the other parts are \emph{long} induced cycles. The decomposition has the same 1-degeneracy condition as the one in~\cite{conlon2022size}. The new part---the graph with bounded treewidth---is more complex, and thus requires different techniques to be embedded. In particular, it is known that bounded degree graphs with bounded treewidth have linear size-Ramsey numbers~\cite{berger2019size,kamcev2021size}, hence it is possible to construct a host graph on just a linear number of edges to embed the part of bounded treewidth from our decomposition.
	
	It may seem like we are close to being done now, since we know how to handle each of the parts in the decomposition of $H$. But this is not really the case, as there are many difficulties to overcome. The main challenge lies in the fact that the host graphs we need for induced cycles on the one hand, and for the bounded treewidth part on the other, are very different, as are the embedding strategies used in these two cases. We cannot simply, for example, take a union of these two host graphs, as the adversary may colour one of them in blue and the other in red, and then we could not find all parts in the same colour. Instead, we need to carefully intertwine the two host graphs into a new host graph in such a way that if one part of the decomposition cannot be found in say blue, this guarantees that \emph{all} parts can be found in red.
	
	The host graph for bounded degree bounded treewidth graphs constructed by Kam\v{c}ev, Liebenau, Wood, and Yepremyan~\cite{kamcev2021size} is a blow-up of a third power of a random regular graph, in which each vertex of the random regular graph is replaced by a constant-sized clique, and each edge---by a complete bipartite graph. We make use of a slight modification of this host graph, where we use a blow-up of a binomial random graph $G(n', \frac{\log{n'}}{n'})$ instead, and take many copies of it (each with fresh randomness), superimposed in a particular way, as our host graph. Namely, we first sample a graph $G \sim G^C(n,p)$ with $p = n^{-1/2 + \delta}$, and then partition almost all of its cliques of size $C$ into a number of almost perfect packings (disjoint cliques which cover almost all vertices), see Lemma~\ref{lem:partition-into-matchings}. On top of each such packing, we add a freshly generated copy of a blow-up of the third power of $G(n', \frac{\log{n'}}{n'})$, such that the blow-up of each vertex of $G(n', \frac{\log{n'}}{n'})$ is mapped to one of the cliques in the packing. The union of $G^C(n,p)$ with all these blow-ups of third powers of $G(n', \frac{\log{n'}}{n'})$ is our host graph $\Gamma$ (see Definition~\ref{defin:the-a-i-s}).
	
	We prove that if the decomposition part of bounded treewidth is not present in blue in the host graph $\Gamma$, it is present in red (Lemma~\ref{lem:universal-graph-existence}) \emph{and} we can find 21 linear-sized sets, where each two of them form a dense red regular pair (Proposition~\ref{prop:q-partition-or-cliques}, Lemmas~\ref{lem:density-from-cliques} and~\ref{lem:regular-clique-complete-bipartite-graphs}). We then use those sets to embed the long induced cycles in red similarly to before (Theorem~\ref{thm:embed-in-regular-pairs}). 
    Let us note that finding the 21 sets with the required density is another important and technically involved part of the proof, relying on Tur\'an's theorem, results from \cite{kamcev2021size}, and a number of careful counting arguments.

	With all this at hand, we give a high level overview of the proof of Theorem~\ref{thm:main}. We start by decomposing the graph $H$ into an ordered collection of one bounded treewidth induced subgraph $\mathcal{T}$ and a number of induced cycles of length at least $L$ for a big constant $L$ (Lemma~\ref{lem:decomposition}). Additionally, we make sure that each vertex has the property that it has at most degree $1$ to the previous subgraphs in the collection.
	
	We have different approaches for embedding $\cT$, and for embedding the induced cycles. We embed the subgraphs from the decomposition one by one in order, starting from $\cT$. We denote our host graph by $\Gamma$.  Consider the largest subset $U$ of the vertex set of $\Gamma$ such that either the blue or the red subgraph of $\Gamma[U]$ does not contain $\cT$. We distinguish two cases depending on whether the set $U$ is larger than $\iota n$ for a small constant $\iota$ or not. The constant $\iota$ is chosen such that in the latter case $\cT$ can be embedded in one of the sets given by an application of the sparse regularity lemma.
	
	In the latter case, we apply the regularity lemma to the coloured $\Gamma$ to obtain 21 sets which are pairwise $(\varepsilon,p)$-regular and where each pair has density at least $\gamma p$ in say red for appropriate constants $\varepsilon$ and $\gamma$. After an appropriate cleaning process (Lemmas~\ref{lem:cleanup} and~\ref{lem:typical_vertices}) which ensures that the red neighbourhood of each vertex behaves nicely, those sets will still be much larger than $U$. We then embed $\cT$ into one of those 21 sets in red, which is particularly convenient since this way the neighbourhoods of the vertices of $\cT$ into the other 20 sets (in red) are large. This enables us to successfully embed the remaining parts from the decomposition. Namely, what is left is to use the remaining 20 sets to embed the induced cycles. The candidate sets for each vertex in those induced cycles are of size at least of order $np$ each. Now we use the technique developed in \cite{conlon2022size} to embed the cycles in the regular pairs (Theorem~\ref{thm:embed-in-regular-pairs}).
	
	In the former case, we have a reasonably large subgraph $\Gamma':=\Gamma[U]$ with no copy of $\cT$ in one of the two colours, say blue. We then use the result from~\cite{kamcev2021size} to show that, since there is no blue copy of $\cT$ in $\Gamma'$, then a red copy must exist in each large subgraph of $\Gamma'$ (Lemma~\ref{lem:universal-graph-existence}). Now we apply the sparse regularity lemma to the red subgraph of $\Gamma'$, and using the fact that there are no blue copies of $\cT$ in $\Gamma'$, we conclude as discussed above that there exists a collection of 21 linear-sized sets, all pairs of which are regular and have enough density in red, which is the most technical part of our proof (Proposition~\ref{prop:q-partition-or-cliques}, Lemmas~\ref{lem:density-from-cliques} and~\ref{lem:regular-clique-complete-bipartite-graphs}). Finally, we embed the parts of $H$ as before, by first embedding $\cT$ in one of the linear-sized sets, and then embedding the induced cycles in the remaining $20$ sets.

	\section{Preliminaries}
	\label{sec:preliminaries}
	In this section, we introduce our notation and state and prove some results used in our proof. 
	\newline
	
	\noindent\textbf{Notation.} We use standard graph and set theoretic notation. For a graph $G = (V,E)$ and not necessarily disjoint vertex sets $A,B \subseteq V$, we denote by $e_G(A,B)$ the number of edges with one endpoint in $A$ and another endpoint in $B$ (edges with both endpoints in $A\cap B$ are counted only once). The neighbourhood $N_G(v,A)$ of a vertex $v$ in $A$ is the set of vertices adjacent to $v$ in $G$, and the cardinality, $d_G(v,A)$, of that set is referred to as the degree of $v$ in $A$. We also write $N_G(v)$ for $N_G(v,V(G))$ and $d_G(v)$ for $d_G(v,V(G))$. For an integer $k\geq 1$, we denote by $G^k$ the graph obtained by adding an edge between every two vertices that are at distance at most $k$ in the graph.  
 An $r$-uniform hypergraph $\cH=(V,E)$ consists of a vertex set $V$ together with a collection $E$ of $r$-subsets of $V$ called hyperedges. For a subset $S \subseteq V$, we define the degree $d_{\cH}(S)$ to be the number of hyperedges in $\cH$ which contain $S$.
	
	For simplicity, we employ the following conventions. We omit rounding of real numbers to nearest integers whenever it is not of vital importance. For two constants $a,b$, we use $a \ll b$ to indicate that $b$ is large enough as a function of $a$ so that our proofs go through. For example, we often use inequality chains like $a \gg b \gg c \gg d$, which also implies that in particular $a \gg bcd$. Furthermore, we write $\{a_1,\dots,a_k\} \gg \{b_1,\dots,b_m\}$ to abbreviate that $a_i \gg b_j$ for all $i$ and $j$. For two functions $f,g:\mathbb N\rightarrow \mathbb R$, we write $f \approx g$ to express that $\lim_{n\rightarrow \infty}\frac{f(n)}{g(n)}=1$. We denote by $I_C$ the independent set on $C$ vertices and by $K_C$ the clique on $C$ vertices.

\subsection{Concentration inequalities}
\label{subsec:concentrationinequalities}
We make use of the following standard concentration bounds for random variables.

\begin{theorem}[McDiarmid's inequality, \cite{mcdiarmid1989method}]\label{thm:mcdiarmid}
Consider the product $(\Omega,Pr)$ of $N$ probability spaces $(\Omega_1, Pr_1),$ $ \dots, (\Omega_N, Pr_N)$. For a random variable $X:\Omega \rightarrow \mathbb{R}$, the \emph{effect} of the $i$-th coordinate is defined to be at most $c$ if for every pair $\omega, \omega' \in \Omega$ that agree on all but the $i$-th coordinate, it holds that $\left|X(\omega) - X(\omega')\right| \leq c$. Let $X: \Omega \rightarrow \mathbb{R}$ be a random variable such that for each $i\in [N]$, the effect of the $i$-th coordinate on $X$ is at most $c_i$. Then for all $t \geq 0$, we have
$$ Pr\Big[\left|X- \mathbb{E}[X]\right|\geq t\Big] \leq e^{- \frac{2t^2}{\sum_{i=1}^N c_i^2}}.$$
\end{theorem}

The next theorem is a form of Chernoff's inequality.
\begin{theorem}[Theorem A.1.19 in~\cite{alon2016probabilistic}]
\label{thm:weighted-Chernoff}
For every $C>0$ and $\varepsilon > 0$, there exists $\delta > 0$ so that the following holds: Let $X_i$, $1 \leq i \leq n$ for an arbitrary $n$, be independent random variables with $\mathbb{E}[X_i] = 0$, $\left|X_i\right| \leq C$ and $Var[X_i] = \sigma_i^2$. Set $X = \sum_{i=1}^n X_i$ and $\sigma^2 = \sum_{i=1}^n \sigma^2_i$ so that $Var[X] = \sigma^2$. Then for $0 < a \leq \delta \sigma$, it holds that
$$Pr[X > a\sigma] < e^{- \frac{a^2}{2}(1- \varepsilon)}.$$
\end{theorem}

We will also need the following slightly altered version of Chernoff's inequality.
\begin{lemma}[Chernoff bound, weighted version]
\label{lem:weighted-Chernoff}
Let $C,\gamma>0$ and let $X_i$ for $1\leq i \leq n$ be independent random variables with $X_i = C_i$ with probability $p$ and $X_i = 0$ otherwise, where $0< C_i \leq C$ and $0 < p \leq \frac{1}{2}$. Then for $X := \sum_{i=1}^n X_i$ and $\mathbb{E}[X]\rightarrow \infty$, we have 
$$ Pr\Big[\left|X - \mathbb{E}[X]\right| > \gamma \mathbb{E}[X]\Big] \leq e^{-\Theta(\mathbb{E}[X])}.$$
\end{lemma}
\begin{proof}
Note that $\mathbb{E}[X] = \sum_{i=1}^n C_i p = \Theta(np)$. Let $Y_i := X_i - pC_i$ and $Z_i = -Y_i$ so that $Y_i,Z_i$ fulfill the conditions of Theorem~\ref{thm:weighted-Chernoff}, setting $Y := \sum_{i=1}^n Y_i$ and $Z := \sum_{i=1}^n Z_i$. We have $Var[Y_i] = Var[Z_i] = \mathbb{E}[Y_i^2] = \mathbb{E}[Z_i^2] = C_i^2 (1-p)p$. Thus $Var[Y]=Var[Z] = \sum_{i=1}^n C_i^2(1-p)p = \Theta(np)$ and $\sigma = \Theta(\sqrt{np})$. We pick $a = \min \{\frac{\gamma \mathbb{E}[X]}{\sigma}, \sigma \delta_{\ref{thm:weighted-Chernoff}} \} = \Theta(\sqrt{np})$ where $\delta_{\ref{thm:weighted-Chernoff}}$ is the $\delta$ given by Theorem~\ref{thm:weighted-Chernoff} applied with $C:=\max_{i\in[n]}\{C_i\}$ and $\varepsilon := \frac{1}{2}$. Thus, by Theorem~\ref{thm:weighted-Chernoff}, we get
$$ Pr\Big[X > \mathbb{E}[X](1+\gamma)\Big] \leq Pr\Big[Y > a\sigma\Big] \leq e^{-\Theta(np)} = e^{-\Theta(\mathbb{E}[X])}$$
$$ Pr\Big[X < \mathbb{E}[X](1-\gamma)\Big] \leq Pr\Big[Z > a\sigma\Big] \leq e^{-\Theta(np)} = e^{-\Theta(\mathbb{E}[X])}.$$
\end{proof}

\subsection{Regularity method}
	\label{subsec:regularitymethod}
	 One of the main tools we use in our proof is a sparse version of Szemer\'edi's regularity lemma, and to state it we need the following two definitions.
	
	\begin{defin}\label{def:regularity}
  For a graph $G$ and disjoint subsets $V_1, V_2 \subseteq V(G)$, the pair $(V_1, V_2)$ is said to be \emph{$(\varepsilon, p)$-regular} for some $0 <
  \varepsilon, p \le 1$ if, for every $U_1 \subseteq V_1$, $U_2 \subseteq V_2$ with
  $\left|U_1\right| \ge \varepsilon\left|V_1\right|$, $\left|U_2\right| \geq \varepsilon\left|V_2\right|$, it holds that
  \[
    \left|d_G(U_1, U_2) - d_G(V_1,V_2)\right| \leq \varepsilon p,
  \]
  where $d_G(A, B) = e_G(A, B)/(|A||B|)$ is the density of the pair $(A,B)$.
\end{defin}

\begin{defin}\label{def:uniformity}
A graph $G = (V,E)$ is said to be \emph{$(\gamma,p)$-upper-uniform} if for all $U, W \subseteq V$ with $U \cap W = \varnothing$ and $|U|, |W| \geq \gamma |V|$, $e_G(U,W) \leq (1+\gamma)p|U||W|$.
\end{defin}

The next standard lemma follows directly from Definition~\ref{def:regularity} and shows that large enough subsets of regular pairs still constitute regular pairs.
\begin{lemma}\label{lem:large-subset-reg-inheritance}
  Consider constants $0 < \varepsilon_1 < \varepsilon_2 \leq 1/2$ and $p \in (0, 1)$, and an $(\varepsilon_1, p)$-regular pair $(X, Y)$.
  Every two subsets $X' \subseteq X$ and $Y' \subseteq Y$ of size $|X'|
  \geq \varepsilon_2|X|$ and $|Y'| \geq \varepsilon_2|Y|$ constitute an $(\varepsilon_1/\varepsilon_2, p)$-regular
  pair with $d(X',Y') \in d(X, Y) \pm \varepsilon_1 p$.
\end{lemma}

In what follows, we state the sparse regularity lemma, which is an adaptation of Szemer\'edi's regularity lemma for sparse graphs. For a set $S$, we call $S_1, \dots, S_t$ an \emph{equipartition} of $S$ if for all $i, j \in [t]$, we have $|S_i| = |S_j| \pm 1$. We call an equipartition $V_1, \dots, V_t$ of the vertices of some graph $G$ an \emph{$(\varepsilon,p)$-regular equipartition} if all but at most $\varepsilon \binom{t}{2}$ pairs $(V_i, V_j)$ are $(\varepsilon,p)$-regular.
\begin{theorem}[Sparse regularity lemma, \cite{kohayakawa1997szemeredi,kohayakawa2003szemeredi}]
\label{thm:sparse-regularity-lemma}
For every $\varepsilon >0$ and every integer $t_0>0$, there is $\gamma>0$ and an integer $T \geq t_0$ such that every $(\gamma,p)$-upper-uniform graph $G$ admits an $(\varepsilon,p)$-regular equipartition $V_1, \dots, V_t$ of its vertices with $t_0 \leq t \leq T$.
\end{theorem}

Note that the (sparse) regularity lemma was originally stated to give an exceptional set $V_0$ of size at most $\varepsilon |V(G)|$ as part of the equipartition. One function of this exceptional set is to be able to take all other sets to be of precisely the same size. As we do not need that and can afford to instead have differences of one in size between the parts, we can distribute the exceptional set between the other sets as equally as possible. The regularity property then still holds, but with a larger $\varepsilon$.

The next two lemmas help us `clean up' a regular partition in such a way that all vertices have large neighbourhoods in each set of the regular partition and all pairs of neighbourhoods that belong to regular pairs are also regular. Note that in Lemma~\ref{lem:typical_vertices}, we replaced $\tilde{n}d/4$ with $\tilde{n}d/20$, which is just a constant change that can be compensated for in the proof in~\cite{conlon2022size} by taking $\beta$ even smaller.
\begin{lemma}[Lemma 3.3 in~\cite{conlon2022size}]\label{lem:cleanup}
  For every $\Delta \in \mathbb{N}$ and $\gamma > 0$, there exists $\varepsilon_0 > 0$ such
  that the following holds for any $0 < \varepsilon \le \varepsilon_0$ and $p \in (0,1)$. Let
  $H$ be a graph with maximum degree $\Delta$ and let $\{V_i\}_{i \in V(H)}$ be
  a family of subsets of some graph $G$ such that $(V_i,V_j)$ is
  $(\varepsilon,p)$-regular of density $d \geq \gamma p$ (with respect to $G$) for
  every $ij \in H$. Then, for every $i \in V(H)$, there exists $V_i' \subseteq
  V_i$ of order $|V_i'| \ge (1-\Delta\varepsilon)|V_i|$ such that $d_G(v, V_j') \geq
  d|V_j|/2$ for every $v \in V_i'$ and all $ij \in H$.
\end{lemma}
\begin{lemma}[Lemma 3.5 in~\cite{conlon2022size}]
  \label{lem:typical_vertices}
  For all $\varepsilon', \alpha, \gamma, \beta > 0$, there exist $\varepsilon_0 = \varepsilon_0(\varepsilon',\gamma,\beta)$ and $K = K(\varepsilon', \alpha, \gamma)$ such
  that, for every $0 < \varepsilon \le \varepsilon_0$ and $p \ge K(\log n/n)^{1/2}$, the
  random graph $\Gamma \sim \Gnp$ w.h.p.\ has the following property.

  Suppose $G \subseteq \Gamma$ and $V_1, V_2 \subseteq V(\Gamma)$ are disjoint
  subsets of order $\tilde n = \alpha n$ such that $(V_1, V_2)$ is
  $(\varepsilon,p)$-regular of density $d \ge \gamma p$ with respect to $G$. Then there
  exists $B \subseteq V(\Gamma)$ of order $|B| \le \beta \tilde n$ such that for each
  $v, w \in V(\Gamma) \setminus (V_1 \cup V_2 \cup B)$ (not necessarily
  distinct) the following holds: for any two subsets $N_v \subseteq N_\Gamma(v,
  V_1)$ and $N_w \subseteq N_\Gamma(w, V_2)$ of order $\tilde n d/20$, both
  $(N_v, V_2)$ and $(N_v, N_w)$ are $(\varepsilon', p)$-regular of density $(1 \pm
  \varepsilon')d$ with respect to $G$.
\end{lemma}

\subsection{Embedding cycles into a regular partition}

The following lemma, which we borrow from \cite{conlon2022size}, is used to embed long induced cycles into regular pairs.
\begin{lemma}[Lemma 4.2 in~\cite{conlon2022size}]\label{lem:cycle-embedding-lemma}
  For every $\alpha, \gamma > 0$ and  $\ell \geq 3$, there exist
  $c(\alpha,\gamma, \ell),$ $\varepsilon_0(\gamma, \ell) > 0$, and $K(\alpha, \ell, \gamma)>0$ such that, for
  every $0 < \varepsilon \leq \varepsilon_0$ and $p \geq Kn^{-(\ell-2)/(2\ell-3)}$, the random graph
  $G \sim \Gnp$ w.h.p.\ has the following property.

  Let $C$ be a cycle of length $t \in [\ell, cn]$. Let
  $G' \subseteq G$ and, for each $v \in V(C)$, let $s_v \in V(G)$ be a
  uniquely chosen vertex. Then, for any collection of subsets $N_v \subseteq N_{G'}(s_v)$ of order
  $\alpha n p$ such that $(N_v, N_w)$ is $(\varepsilon, p)$-regular of density $d \ge
  \gamma p$ (with respect to $G'$) for each $vw \in C$, there exists a copy of
  $C$ in $G'$ which maps each $v \in V(C)$ to $N_v$.
\end{lemma}

In what follows, we make some minor modifications to the proof of the main theorem in~\cite{conlon2022size}, to adapt it to our application. The goal is, similarly to~\cite{conlon2022size}, to embed a `$1$-degenerate' collection of long induced cycles into a host graph with $20$ vertex sets, all pairs among which are regular and dense enough. The main difference is that in our case some vertices already have predefined candidate sets before the embedding process begins. The proof remains almost unchanged, but we provide it here for completeness.
\begin{theorem}[Proof of Theorem 1.2 in \cite{conlon2022size}]\label{thm:embed-in-regular-pairs}
For every $c, \delta, \gamma, \varepsilon, \mu > 0$ and $K(\mu, \ell, \gamma)$ such that 
$\varepsilon \ll \gamma<1$, and $c \ll \{\mu, \gamma, \delta\}<1$, and $\{c^{-1}, \varepsilon^{-1}\} \gg \ell\geq 3$ such that $ \delta = \frac{1}{4\ell-6}$, the following holds
w.h.p.\ for $G \sim G(n,p)$ with
$p \geq K n^{-\frac{1}{2}+\delta}$. Let
$V_1, \dots, V_{20}$ be disjoint subsets of
$V(G)$, each of size $|V_i| \geq \tilde{n} := \mu n$, and let $G'$ be a subgraph of $G$.

Let $F$ be a cubic graph of $cn$ vertices with vertex partition $F_1 \cup \dots \cup F_f$, such that each $F_i$ is an induced cycle of length at least $\ell$, and such that for each $v \in F_i$, either $v$ has precisely one neighbour in $F_1 \cup \dots \cup F_{i-1}$, or $v$ has no neighbours in $F_1 \cup \dots \cup F_{i-1}$, but there is a unique vertex $u_v\in V(G)-(V_1\cup\ldots \cup V_{20})$ and a "candidate set" $C_v \subseteq N_{G'}(u_v)$ for $v$ such that for all $j$ we have $|C_v \cap V_j| \geq \frac{\tilde{n}d}{4}$, where $d:= \gamma p$. In addition, suppose no $u_v$ is chosen more than three times. Suppose also that

\begin{enumerate}
    \item $d_{G'}(v,V_i) \geq \frac{\tilde{n} d}{4}$ for each $v \in V_j$ and $i \neq j \in [20]$,
    \item for all distinct $i,j,h,g$ (but possibly $h=g$), for each $v \in V_h, w \in V_g$, and each $a,b \in V(F)$ with associated $C_a, C_b$, and any
    \begin{itemize}
    \item $N_1 \subseteq N_{G'}(v, V_i)$ or $N_1 \subseteq C_a \cap V_i$ and
    \item $N_2 \subseteq N_{G'}(w,V_j)$ or $ N_2 \subseteq C_b \cap V_j$
    \end{itemize}
     of size $|N_1|=|N_2|= \frac{\tilde{n}d}{20}$,
     $(N_1, N_2)$ and $(N_1, V_j)$ are $(\varepsilon,p)$-regular of density at least $\frac{d}{2}$ in $G'$.
\end{enumerate}
Then there exists a copy of $F$ in $G'[V_1 \cup \dots \cup V_{20}]$, where each vertex is mapped to its candidate set $C_v$, if it has one assigned to it.
\end{theorem}
\begin{proof}
For each $i\leq 10$, denote by $V_i^0$ the set $V_i$, and denote by $V_i^1$ the set $V_{i+10}$.
Before we start, assign to each vertex $v$ in $F$ a number $\varphi(v)$ from $\{1,\ldots,10\}$, such that each two vertices at distance at most $2$ in $F$ get a different number. This can be done easily by a greedy assignment.
Now we embed the parts $F_1,\ldots,F_f$ into $G'$, one at a time, in the given order. Suppose we already embedded $F_1,\ldots, F_{i-1}$ and let us show how to embed $F_i$.

  \begin{enumerate}[label=({\roman*}), leftmargin=3em]
    \item\label{proc-neighbour}
    For every vertex $v\in F_i$, we define the vertex $u_v\in G'$ as follows. If $v$ has a neighbour $a_v$ in $F_1,\ldots F_{i-1}$, we set $u_v$ to be the image of that neighbour in $G'$. Otherwise, $u_v$ is the vertex $u_v$ from the statement of the theorem. 
    
    \item\label{proc-candidate} 
    Let $b_v$ be the smallest number in $\{0,1\}$, such that the set $V^{b_v}_{\varphi(v)}$ has at least $\frac{\tilde{n}d}{20}$ vertices in $N_{G'}(u_v)$ not occupied by the embedding of $F_1,\ldots, F_{i-1}$. If there is no such $b_v$, stop the procedure, and otherwise let $S_v$ be the set of these at least $\frac{\tilde{n}d}{20}$ many non-occupied vertices.

    \item\label{proc-embed} 
    Now, since $|S_v| \geq\tilde n d/20$ for every $v\in F_i$, we have that for those sets, using the second property from the theorem, the conditions of Lemma~\ref{lem:cycle-embedding-lemma} are satisfied, so there exists a copy of the induced cycle $F_i$ in $G'$, which maps every vertex $v\in F_i$ to $S_v$ (note that here we used that for every $w$ adjacent to $v$ in $F_i$, it holds that $\varphi(a_v)\neq \varphi(w)\neq \varphi(v)$).

  \end{enumerate}
 If our procedure did not stop in Step~\ref{proc-candidate}, we found the required copy of $F$ in $G'$, so it is enough to show that we did not stop early.
 
 Suppose we are at the point of the algorithm where we want to embed $F_i$, and let us show that we can successfully do that.
  Since $F$ has $cn$ vertices, the set $X$ of all occupied vertices (at that moment) in $V_1^0 \cup \dotsb \cup V_{10}^0$ is clearly also of size at most $cn$. We denote by $B_1$ the set of vertices $v \in F_1 \cup \dots \cup F_{i-1}$ for which $b_v=1$, and denote $U = \{u_v :v\in B_1\}$; note that $|B_1|$ is the number of vertices which are embedded into sets $V_1^1,\ldots, V_{10}^1$. 
  If we now show that $|B_1|=O(1/p)$, then we would be done, as every vertex $u_v$ has at least $\tilde n d/4$ neighbours in each $V_i^1$, so our procedure would not stop early.
  
  Suppose for contradiction that $|B_1|\geq 3C/p $ for a sufficiently large constant $C$, which implies also that $|U|\geq C/p$. Hence, since each vertex in $U$ has at least $\tilde n d/8$ neighbours in $X$, and accounting for possible double-counting, we get that $e_{G'}(U,X)\geq \frac{|U|\tilde n d}{16}$. Meanwhile, using a Chernoff bound, a union bound and that $C$ is sufficiently large, we conclude that w.h.p.\ it holds that for every set $U'$ on at least $C/p$ vertices and $X'$ on $cn$ vertices, there are at most $2|U'||X'|p$ edges in $G$ between $U'$ and $X'$. Therefore, by considering a superset $X'$ of size $cn$ of our set $X$, we get that $2|U|cnp\geq e_{G'}(U,X')\geq \frac{|U|\tilde n d}{16}$, which gives the required contradiction since $c \ll \mu \gamma$ (recalling that $\tilde n=\mu n$ and $d=\gamma p$).
  
\end{proof}

\subsection{Other auxiliary and classic results}
\label{subsec:auxiliaryclassic}
In this subsection, we state some graph theoretic results which will come in handy in our proof.

The first lemma guarantees that almost perfect hypergraph matchings preserve some properties when a subset of their elements is considered.

\begin{lemma}\label{lem:proportion of A} 
Let $M$ be a hypergraph matching on the vertex set $[n]$, where each edge is of size $C$, and $M$ covers at least $(1-\gamma)n$ vertices. If $S\subseteq [n]$ is of size at least $|S|\geq 4\gamma n$, then there exist at least $|S|/2C$ edges $h\in M$ with $|h\cap S|\geq \gamma C/2$.
\end{lemma}

\begin{proof}
Suppose for contradiction that for some set $S$, at most $|S|/2C$ edges satisfy $|h\cap S|\geq \gamma C/2$. Then the number of vertices in $S$ is at most
$$
    \frac{|S|}{2C}\cdot C+\frac{n}{C}\cdot \frac{\gamma C}{2}+\gamma n=\frac{|S|}{2}+\frac{3n\gamma}{2}<|S|
$$
where the first term bounds the number of vertices of $S$ in edges $h$ with $|h\cap S|\geq \gamma C/2$, the second one the vertices in edges which do not satisfy this condition, and the third one counts the vertices in $S$ which are not in an edge in $M$. This gives the required contradiction.
\end{proof}

Next, we state the classic theorem of Tur\'an.
\begin{theorem}[Tur\'an's theorem,~\cite{turan1941external}]
Suppose $G$ is a graph on $n$ vertices with no $K_{r+1}$ as a subgraph. Then $|E(G)| \leq (1-\frac{1}{r})\frac{n^2}{2}$.
\end{theorem}

We also make use of another well-known extremal result for bipartite graphs. 
\begin{theorem}[K\"ov\'ari, S\'os, Tur\'an,~\cite{kovari1954problem}]
\label{thm:kovari-sos-turan}
The maximum number of edges in an $n$-vertex graph with no $K_{\ell,\ell}$ subgrpah is less than $(\ell-1)^{\ell} n ^{2 - 1/\ell} + \ell n + 1$.
\end{theorem}

The next lemma concerns partitioning the edges of a hypergraph into matchings. The chromatic index $q(\mathcal{H})$ of a hypergraph $\mathcal{H}$ is the smallest integer $q$ such that the set of edges $\mathcal{H}$ can be partitioned into $q$ matchings.
\begin{lemma}[\cite{pippenger1989asymptotic}]
\label{lem:chromatic_index}
For an integer $r \geq 2$ and $\gamma > 0$, there exists $\beta=\beta(r,\gamma)>0$ so that the following holds. If an $r$-uniform hypergraph $\mathcal H$ has the following properties for some $t$:
\begin{enumerate}
    \item $(1-\beta)t < d(v) < (1+\beta)t$ holds for all vertices $v$,
    \item $d(u,v) < \beta t$ for all distinct pairs of vertices $u,v$,
\end{enumerate}
then $q(\mathcal H) \leq (1+\gamma)t$.
\end{lemma}

\section{Graph decomposition}
\label{sec:decomposition}
In this section, we show that the vertices of every cubic graph $G$ can be decomposed into sets which induce long cycles and other bounded treewidth graphs, in a way convenient for the embedding we use to prove Theorem~\ref{thm:main}.

We start by stating the following result from \cite{Kosowski2015chordal}, which shows that every graph without long induced cycles and with bounded maximum degree has bounded treewidth. This result (albeit with weaker constants) had already been shown in~\cite{bodlaender1997treewidth}.

	\begin{lemma}[\cite{Kosowski2015chordal}]
	\label{lem:chordal-bounded-treewidth}
	Any graph $G$ without induced cycles of length at least $k$ and with maximum degree $\Delta$ has treewidth at most $(k - 1)(\Delta - 1) + 2$.
	\end{lemma}

The \emph{strong product} $G \boxtimes H$ of graphs $G$ and $H$ is the graph with vertex set $V(G) \times V(H)$ in which $(v_1, u_1)$ is adjacent to $(v_2,u_2)$ if $v_1=v_2$ and $\{u_1,u_2\}\in E(H)$, or $\{v_1,v_2\} \in E(G)$ and $u_1=u_2$, or $\{v_1,v_2\} \in E(G)$ and $\{u_1,u_2\}\in E(H)$. When $H$ is a complete graph, we refer to $G\boxtimes H$ as a blow-up of $G$.

The following lemma states that every graph of bounded treewidth and bounded maximum degree is contained in a sufficiently large blow-up of a tree.
\begin{lemma}[\cite{ding1995some,wood2009tree}]
\label{lem:bounded-degree-and-treewidth-tree-blow-up}
Let $G$ be a graph of treewidth $w$ and maximum degree $d$. Then $G$ is a subgraph of $T \boxtimes K_{18wd}$ for some tree $T$ with maximum degree $18wd^2$.
\end{lemma}

To state our decomposition result, it will be convenient for us to have the following definition.
\begin{defin}
  Let $G$ be a graph and let $\mathcal{S}=\{S_1,\ldots S_t\}$ be a partition of its vertex set. Then we say that $\mathcal{S}$ is a \emph{$1$-degenerate} partition of $G$ if every vertex in $S_i$ is adjacent to at most one vertex in $S_{1}\cup\ldots \cup S_{i-1}$ for all $2\leq i\leq t$. 
\end{defin}

We are now ready to state our decomposition result for cubic graphs.
\begin{lemma}\label{lem:decomposition}
Let $\ell \geq 5$ and
let $G$ be a graph with $\Delta(G)\leq 3$. Then $G$ admits a decomposition into subgraphs $J$ and  $F_1, \dots, F_g$ such that $V(J), V(F_1), \dots, V(F_g)$ is a 1-degenerate partition of $V(G)$ and the following hold:
\begin{enumerate}[label=(\Alph*)]
    \item Each $F_i$ is an induced cycle $C_{L}$ with $L \geq \ell$.
    \item $J$ has treewidth at most $2\ell$ and is hence a subgraph of $T \boxtimes K_{400\ell}$ for some tree $T$ of maximum degree $400\ell$. \label{prop:treewidth}
\end{enumerate}
\end{lemma}
\begin{proof}
To obtain the graphs $F_1,\ldots, F_g$, take out induced cycles of length at least $\ell$ from $G$ one by one, each time removing all vertices from $G$ which lie on the removed cycle, and taking a new induced cycle in the obtained graph. When there are no more long induced cycles to be removed, we are left with a subcubic graph $J$ which contains no induced cycle of length at least $\ell$, and hence by Lemma~\ref{lem:chordal-bounded-treewidth} has treewidth bounded by $2\ell$.
By Lemma~\ref{lem:bounded-degree-and-treewidth-tree-blow-up}, $J$ is also a subgraph of $T \boxtimes K_{400\ell}$ for some tree $T$ of maximum degree $400\ell$.
Notice that $V(J), V(F_1), \dots, V(F_g)$ is indeed $1$-degenerate, as each vertex in some cycle $F_i$ can have at most one edge going out of $F_i$, while $J$ is the first part of the decomposition, so there are no additional conditions for it.
\end{proof}

\section{The host graph}\label{sec:host}
In this section we describe our random graph model, and prove several results about its properties. Our model consists of a number of random cliques and random complete bipartite graphs, chosen in a particular way. 

\subsection{The cliques}
In order to describe our random graph model, we first recall the by now standard definition of a \emph{Steiner system}. A Steiner system with parameters $t, k, n$, denoted by $S(t,k,n)$, is an $n$-element set $S$ together with a collection of $k$-element subsets of S (called blocks) with the property that each $t$-element subset of S is contained in exactly one block. In particular, we will use Steiner systems with $t=2$; by a well known result of Wilson \cite{wilson1975existence} we have that for every large enough $n$, if $k(k-1)$ divides $n-1$, then a $S(2,k,n)$ exists. This immediately implies the following.

\begin{cor}[\cite{wilson1975existence}]
For any integer $C>0$, and large enough $n$, with $n-1$ divisible by $C(C-1)$, there exists a partition of the edges of $K_n$ into cliques of size $C$.
\end{cor}

To define our random graph model, for each $C$ and $n$ as in the corollary above we fix one (arbitrary) Steiner system $S(2,C,n)$ and call it a \emph{canonical} Steiner system. In the rest of the paper, we denote by $S(2,C,n)$ the set of cliques (which we also refer to as blocks) in $K_n$ from the canonical $S(2,C,n)$. Sometimes we treat a block $B\in S(2,C,n)$ as a set of its vertices. We refer to the edges in a block $B$ as $E(B)$.
We are now ready to define our random graph model.
\begin{defin}
\label{defin:gcnp}
Given $p$, and $C$ and $n$ as above, let $G^C(n,p)$ be the random graph obtained by, independently for each block in a canonical $S(2,C,n)$, including all edges induced by that block with probability $p$.
\end{defin}
Note that $G^2(n,p)$ has the same distribution as $G(n,p)$. Also note that by definition each edge in $K_n$ is in precisely one block of $S(2,C,n)$. From now on we assume $n$ is chosen large enough and so that $C(C-1)$ divides $n-1$.

 The next result is needed in order to apply the sparse regularity lemma to $G^C(n,p)$. Recall that a graph $G = (V,E)$ is said to be \emph{$(\gamma,p)$-upper-uniform} if for all $U, W \subseteq V$ with $U \cap W = \varnothing$ and $|U|, |W| \geq \gamma |V|$, $e_G(U,W) \leq (1+\gamma)p|U||W|$ (Defintion~\ref{def:uniformity}).
\begin{lemma}
\label{lem:upper-uniformity}
For $0<\gamma,\delta<1/2$, there exists $K>0$ such that the following holds. The graph $G\sim G^C(n,p)$ with $p \geq K/n$ is w.h.p.\ $(\gamma, p)$-upper-uniform.
\end{lemma}
\begin{proof}
Let $U$ and $W$ be disjoint subsets of $V(G)$ of size at least $\gamma n$. Let $\cB$ be the collection of blocks in $S(2,C,n)$ with at least one vertex in both $U$ and $W$. For each block $B$ in $\cB$, let $X_B$ be the random variable counting the edges in $G$ between $U$ and $W$ contained in $B$. Note that $X_B = |U \cap B| \cdot |W \cap B|$ with probability $p$ and $X_B=0$ otherwise. Therefore, by Lemma~\ref{lem:weighted-Chernoff}, the number of edges between $U$ and $W$ 
$$X:=\sum_{B\in \cB} X_B$$
is at most $(1+\gamma)|U||W|p$ with probability at least $1 - e^{-\Theta(|U||W|p)}$. Indeed, we have $\mathbb{E}[X] = |U||W|p$ and for each block $B$ in $\cB$, it holds that $X_B \leq \binom{C}{2}$, so the conditions of Theorem~\ref{lem:weighted-Chernoff} are satisfied. By a union bound over all exponentially many choices of $U$ and $W$, the probability that $G$ is not $(\gamma,p)$-upper-uniform is at most 
$$ 2^{2n} e^{-\Theta(|U||W|p)} \leq 2^{2n} e^{-\Theta(n^2 p)} = o(1), $$
provided that $K$ is large enough.
\end{proof}

By a similar argument (involving Chernoff bounds) as in the proof of Lemma~\ref{lem:upper-uniformity}, we also have the following.
\begin{lemma}
\label{lem:upper-bound-edges-of-G}
For every $0 < \gamma,\delta < 1/2$, there is a $K>0$ such that the graph $G \sim G^C(n,p)$ with $p \geq K\log{n}/n^2$ w.h.p.\ satisfies $$e(G) \leq (1+\gamma) \binom{n}{2} p \leq n^2p.$$
\end{lemma}

In the proof of Theorem \ref{thm:main}, it will be important for us to consider a random subsampling of $G^C(n,p)$, so that the obtained subgraph has the same distribution as $G(n,\tilde{p})$ for $\tilde{p}\approx \frac{p}{\binom{C}{2}}$.
\begin{defin}
\label{defin:subsample-G-C-n-p}
In the remainder of the paper, we denote by $G$ the random graph $G\sim G^C(n,p)$ for $p=n^{-1/2+\delta}$ where $\delta>0$, and we define $\tilde{G}$ as follows.
For each block $B\in S(2,C,n)$ with $E(B) \subseteq E(G)$, we sample a non-empty subset of the edges $B' \subseteq E(B)$ to be present in $\tilde{G}$ with the following probability
$$ Pr\Big[E(B) \cap E(\tilde{G}) = B'\Big] = \frac{\tilde{p}^{|B'|} (1-\tilde{p})^{\binom{C}{2} - |B'|}}{p} $$
where $\tilde{p}$ is given by $p=1-(1-\tilde{p})^{\binom{C}{2}}$, that is, $\tilde p\approx\frac{n^{\delta - 1/2}}{\binom{C}{2}}$.
\end{defin}

The following simple lemma confirms that $\tilde{G}$ has the same distribution as $G(n,\tilde{p})$.
\begin{lemma}
\label{lem:g-c-n-p-subsample-g-n-p}
$\tilde{G}$ is distributed as $G(n,\tilde{p})$.
\end{lemma}
\begin{proof}
Let $H \subseteq K_n$. We will show that $Pr[\tilde{G} = H] = Pr[G(n,\tilde{p})=H]$. On the one hand, 
$$Pr\Big[G(n,\tilde{p})=H\Big] = \tilde{p}^{e(H)}(1-\tilde{p})^{\binom{n}{2} - e(H)}.$$ 
For each $B \in S(2,C,n)$, denote by $E_B$ the event that $\tilde{G}[B] = H[B]$. Notice that $\tilde{G}=H$ if $E_B$ holds for each $B \in S(2,C,n)$, and that the events $E_B$ are independent. Therefore,
\begin{align*}
Pr\Big[\tilde{G} = H\Big] = & \prod_{B\in S(2,C,n)} Pr\Big[E_B\Big]\\
=& \prod_{\substack{B\in S(2,C,n)\\ e(H[B]) = 0}} (1-p) \prod_{\substack{B\in S(2,C,n)\\ e(H[B]) >0}} Pr\Big[ E_B|E(B) \subseteq E(G)\Big]Pr\Big[E(B) \subseteq E(G)\Big] \\
=& \prod_{\substack{B\in S(2,C,n)\\ e(H[B]) = 0}} (1-\tilde{p})^{\binom{C}{2}} \prod_{\substack{B\in S(2,C,n)\\ e(H[B])>0}}  \frac{ \tilde{p}^{e(H[B])} (1-\tilde{p})^{\binom{C}{2} - e(H[B])} }{p} p\\
=& \tilde{p}^{e(H)} (1-\tilde{p})^{\binom{n}{2} - e(H)},\\
\end{align*}
where the last equality holds since $S(2,C,n)$ is a partition of the edges of $K_n$.
\end{proof}

We also need the following observation.
\begin{lemma}
\label{edge_probability}
Let $G$ be any outcome of $G^C(n,p)$. Then for any $B\in S(2,C,n)$ whose edges are in $G$, and any $e \in E(B)$, we have $Pr[e \in E(\tilde{G})] = \frac{\tilde{p}}{p} \approx \frac{1}{\binom{C}{2}}$.
\end{lemma}
\begin{proof}
We have
\begin{alignat*}{3}
Pr\Big[e \in E(\tilde{G})\Big] &= \sum_{B' \subseteq E(B), e\in B'} Pr\Big[E(B) \cap E(\tilde{G}) = B'\Big] & &=
\sum_{s=1}^{\binom{C}{2}} \binom{ \binom{C}{2}-1 }{s-1} \frac{\tilde{p}^s (1-\tilde{p})^{\binom{C}{2} - s}}{p}\\ &=\frac{\tilde{p}}{p} \sum_{s=0}^{\binom{C}{2}-1} \binom{ \binom{C}{2}-1 }{s} \tilde{p}^s (1-\tilde{p})^{\binom{C}{2} - 1 - s} & &= \frac{\tilde{p}}{p}. 
\end{alignat*}
\end{proof}

The next lemma guarantees that for any equipartition of the vertices of $K_n$ and any choice of a large enough subset $B_0 \subset B$ for each $B \in S(2,C,n)$, there are many sets $B_0$ which are roughly equally distributed between most of the parts.
\begin{lemma} \label{lem:nicely-distributed-cliques}
    Let $t$ be an integer with $t \gg 1$ and let $0 < \beta < 1$. Then for $\rho \ll \beta$, and for $C \gg \{t, \beta^{-1}, \rho^{-1}\}$, the following holds.
    
    Consider $S(2,C,n)$ with $n$ large enough.
    Suppose each $B \in S(2,C,n)$ has an associated $B_0 \subseteq B$ with $|B_0| = (1-\rho)C$.
    For any equipartition of $[n]$ into sets $V_1, \dots, V_t$, there is a collection $\mathcal{B}$ which contains a $(1-\beta)$-fraction of the blocks of $S(2,C,n)$, such that for each $B \in \mathcal{B}$, at least $(1-\beta)t$ sets $V_i$ satisfy $|B_0 \cap V_i| \geq \frac{C}{200t}$.
\end{lemma}
\begin{proof}
    We show the lemma by considering $K_n$, and counting \emph{internal} edges in the $B_0$'s, that is, edges with both endpoints in the same $V_i$, in two ways. Firstly, the number of internal edges is at most $\binom{n/t}{2} t \leq \frac{n^2}{2t}$. Next, we count the internal edges in another way, by summing over two types of blocks---the \emph{bad} blocks $B$, for which there are many sets $V_i$ such that the intersection of $B_0$ and $V_i$ is small, and the remaining \emph{good} blocks.
    
    Consider one block $B \in S(2,C,n)$. For each $i\in[t]$, let $a^B_i = |B_0 \cap V_i|$. Then the number of internal edges in $B_0$ is $\sum_{i=1}^{t} \binom{a^B_i}{2}$ and $C(1-\rho) = \sum_{i=1}^t a^B_i$. We say that a block $B$ is \emph{bad}, if there exists a set of $\beta t$ sets $V_i$ such that $a^B_i < \frac{C}{200t}$; notice that the remaining $(1-\beta)t$ sets must then contain a total of at least $(1-\rho - \beta/200)C$ vertices from $B_0$. The total number of internal edges in a bad block is minimized when each of those $(1-\beta)t$ sets has approximately the same number of vertices from $B_0$ (by the convexity of the binomial coefficient function $\binom{x}{2}$). The number of internal edges for a bad block is hence at least
    $$ (1-\beta) t \binom{ \frac{(1-\rho-\beta/200 )C}{(1-\beta)t} }{2} \geq (1 + \beta/2) \frac{C^2}{2t}.$$
    On the other hand, the number of internal edges for a good block is at least
    $$ t \binom{\frac{(1-\rho)C}{t}}{2} \geq (1-3\rho)\frac{C^2}{2t} $$
    since again the number of internal edges is minimized if all vertices in a good block are distributed equally among the sets $V_i$. Suppose for contradiction that at least $\beta \frac{n(n-1)}{C(C-1)}$ of the blocks in $S(2,C,n)$ are bad. Then the total number of internal edges we get is at least
    $$\beta \frac{n(n-1)}{C(C-1)}\cdot (1 + \beta/2) \frac{C^2}{2t} +(1 - \beta) \frac{n(n-1)}{C(C-1)}\cdot (1-3\rho)\frac{C^2}{2t} >  \frac{n^2}{2t},$$
    contradicting the upper bound on all internal edges. Therefore, at least $(1-\beta)\frac{n(n-1)}{C(C-1)}$ of the blocks in $S(2,C,n)$ are good. These form precisely the desired collection $\mathcal{B}$.
\end{proof}

The following lemma states that, given a set of vertices $R$ in $G$, if a large enough fraction of the blocks present in $G$ intersect $R$ in a large set with no blue clique of a certain size, then $R$ contains 21 linear-sized sets, all pairs of which are regular with large density in red. This serves an important purpose in the proof of Theorem~\ref{thm:main}, ensuring that the long induced cycles from the decomposition of $H$ can be embedded in the same colour as the bounded treewidth part.
\begin{lemma}
\label{lem:density-from-cliques}
Let $C^{-1} \ll \rho \ll \alpha \leq \frac{1}{2}$ and $\{\tau,\varepsilon\} \ll \{C'^{-1},\alpha\}$, as well as $\{C^{-1}, \mu\} \ll \varepsilon$. Let $R \subseteq V(G)$ with $|R| = \alpha n$. Then w.h.p.\ the following holds for $G \sim G^C(n,p)$.

For every red/blue colouring of $E(G)$ such that at least an $\alpha/32$ fraction of the blocks $B$ of $S(2,C,n)$ present in $G$ contain a subset $B' \subseteq B \cap R$ with $|B'| = \rho C$ such that there is no blue $K_{C'}$ in $(R \cap B) - B'$, the following holds. There are disjoint subsets of vertices $V_1,\ldots, V_{21} \subseteq R$, each of size $m \geq \mu n$, such that each pair $(V_i,V_j)$ is $(\varepsilon,p)$-regular in the red subgraph of $G$ and has density at least $\tau p $.
\end{lemma}
\begin{proof}
We start by giving an overview of the proof. We first apply the sparse regularity lemma to the red subgraph of $G[R]$, yielding sets $V_1, \dots, V_t$ such that most pairs of sets are regular. Next, with the help of Lemma~\ref{lem:nicely-distributed-cliques}, we get a collection $\cB$ of many blocks of $S(2,C,n)$ present in $G$, each of which has a large blue $K_{C'}$-free intersection with most sets $V_i$. After that, Claim~\ref{claim:collection_r_blobs} shows via a double-counting argument that there is a collection $\cV \subseteq \{V_1, \dots, V_t\}$ of $r$ sets for a large constant $r$, all pairs of which are regular and such that there are many blocks in $\cB$ that have a large blue $K_{C'}$-free intersection with \emph{all} sets in $\cV$. Now, the absence of a blue $K_{C'}$ and Tur\'an's theorem imply that every $b$-tuple of sets in $\cV$ contains at least one pair of sets with high red density in between (Claim~\ref{claim:no_blue_K_b}). Since $r$ is chosen large enough, the lack of $b$ many sets with pairwise low red density implies by Ramsey's theorem the existence of $21$ sets with pairwise high red density, as required by the statement of the lemma. We continue with the detailed proof.

Let $b = 2C'$, $\beta = \alpha/3200$, and $\xi = \frac{1}{200}$. Choose $r = r(21, b)$ to be the Ramsey number for $21$ and $b$, and note that since $r$ only depends on $C'$, we may assume $r\ll\varepsilon^{-1}$. Let $t_0 \gg \{r,b\}$. Let $T$ be the maximum number of sets given to us by Theorem~\ref{thm:sparse-regularity-lemma} (the sparse regularity lemma) with constants $\varepsilon$ and $t_0$, and note that we can safely assume that $C\gg T$, since $T$ is only a function of $C',\varepsilon$, and we have $C\gg \{C',\varepsilon^{-1}\}$ by assumption.

By a standard application of the Chernoff bound, w.h.p.\ $G$ contains $(1 \pm \beta) \frac{n(n-1)p}{C(C-1)}$ blocks of $S(2,C,n)$. Thus, by assumption, at least $\frac{\alpha}{32}(1 - \beta)\frac{n(n-1)p}{C(C-1)}\geq \frac{\alpha}{64}\frac{n(n-1)p}{C(C-1)}$ blocks $B$ present in $G$ have a subset $B' \subseteq B$ with $|B'|=\rho C$ such that there is no blue $K_{C'}$ in $(R \cap B) - B'$.

Set $G'$ to be the red subgraph of $G[R]$ and $n' := |V(G')| \geq \alpha n$. By Lemma~\ref{lem:upper-uniformity}, for any constant $\gamma>0$, the graph $G$, and therefore also $G'$, is w.h.p.\ $(\gamma,p)$-upper-uniform (as $G'$ has a linear in $n$ number of vertices). We apply Theorem~\ref{thm:sparse-regularity-lemma} to   $G'$, with $\varepsilon$ and $t_0$. We get vertex sets $V_1, \dots, V_t$ with $t_0 \leq t \leq T$ such that each $V_i$ has size $m := \frac{n'}{t}$ and such that all but $\varepsilon \binom{t}{2}$ many pairs $V_i,V_j$ are $(\varepsilon, p)$-regular. Partition $V(G) 
\setminus R$ arbitrarily into $t' \leq t(\frac{1}{\alpha}-1)$ sets $V_{t+1}, \dots, V_{t+t'}$ each also of size $m$.

Next, we apply Lemma~\ref{lem:nicely-distributed-cliques} to $V_1, \dots, V_{t+t'}$ with $t:=t+t'$, $\beta$, and $\rho$, where for each $B$, we either assign $B_0 := B - B'$, where $B'$ is as given above (if it exists, which holds for at least $\frac{\alpha}{64}\frac{n(n-1)p}{C(C-1)}$ blocks $B$), or $B_0$ is an arbitrary subset $B_0 \subset B$ of size precisely $(1-\rho) C$.
We thus get a collection $\mathcal{B} \subseteq S(2,C,n)$ of size at least $(1-\beta)\frac{n(n-1)}{C(C-1)}$ such that for each $B \in \mathcal{B}$, at least $(1- \beta)(t+t')$ sets $V_i$ satisfy $|B_0 \cap V_i| \geq \xi \frac{C}{t+t'}$. At least $(1-2\beta)\frac{n(n-1)p}{C(C-1)}$ of the blocks in $\mathcal{B}$ are present in $G$ by a standard application of the Chernoff bound and a union bound over all partitions $V_1, \dots, V_t$. In other words, among all of the blocks present in $G$, w.h.p.\ at most $3\beta \frac{n(n-1)p}{C(C-1)}$ are not in $\cB$.

Hence, we have at least  $(\frac{\alpha}{64}-3\beta)\frac{n(n-1)p}{C(C-1)} \geq \beta \frac{n(n-1)p}{C(C-1)}$ blocks $B$ from $\cB$ in $G$ that contain a subset $B_0$ of size $(1-\rho) C$ with no blue $K_{C'}$ in $R\cap B_0$. Call these blocks \emph{blue-avoiding}. Note that for each blue-avoiding block $B$, the number of sets $V_i$ among $V_1, \dots, V_t$ with $|B_0 \cap V_i| \geq \xi \frac{C}{t + t'}$ is at least $$(1-\beta) (t+t') - t' = (1-\beta)t - \beta t' \geq  (1-\beta) t - \beta t(\frac{1}{\alpha}-1) \geq \frac{3}{4}t.$$

For a collection $\cV' \subseteq \{V_1, \dots, V_t\}$, we say that a blue-avoiding block $B$ \emph{intersects $\cV'$ nicely} if $|B_0\cap V_i| \geq \xi \frac{C}{t+t'}$ for each $V_i \in \cV'$. We next show a claim that provides us with a collection $\cV \subseteq \{V_1, \dots, V_t\}$ such that many blue-avoiding blocks intersect $\cV$ nicely.
\begin{claim}
\label{claim:collection_r_blobs}
There is a collection $\mathcal{V} \subseteq \{V_1, \dots, V_t\}$ with $|\cV|=r$ such that each pair $V_i, V_j \in \mathcal{V}$ is a regular pair and at least $f \frac{n(n-1)p}{C(C-1)}$ blue-avoiding blocks $B$ intersect $\cV$ nicely, where $f = \frac{\beta}{10^r}$.
\end{claim}
\begin{proof}
Suppose for contradiction there is no such collection $\cV$. Note that the number of $r$-tuples with at least one irregular pair is at most $\varepsilon \binom{t}{2} \binom{t-2}{r-2}$. We count in two ways the number $N$ of pairs $(\mathcal{V}, B)$, consisting of an $r$-tuple $\mathcal{V}$ with no irregular pairs, and a blue-avoiding  block $B$ that intersects $\mathcal{V}$ nicely. On the one hand, each of the blue-avoiding blocks intersects nicely at least $\binom{\frac{3t}{4}}{r}$ many $r$-tuples. On the other hand, we assumed each $r$-tuple with no irregular pairs is intersected nicely by less than $f \frac{n(n-1)p}{C(C-1)}$ blue-avoiding blocks. Therefore,
$$\beta \frac{n(n-1)p}{C(C-1)} \Bigg[\binom{\frac{3t}{4}}{r} - \varepsilon \binom{t}{2} \binom{t-2}{r-2}\Bigg] \leq N \leq \binom{t}{r} f \frac{n(n-1)p}{C(C-1)},$$
and now, using that $(a/b)^b\leq \binom{a}{b}\leq (a e/b)^b$ for all $a\geq b$, and the fact that $t\gg r$, we have
\begin{align*}
\beta \frac{t^r}{r^r} \Bigg[(3/4)^r - \varepsilon e^r r^2\Bigg]  \leq \frac{e^r t^r}{r^r} f \quad \implies\quad 
 \frac{\beta 3^r}{2\cdot 4^r} \leq e^r f,
\end{align*}
where the implication holds as $\varepsilon\ll r^{-1}$. This gives a contradiction by our choice of $f$, showing the existence of the desired collection $\mathcal{V}$.
\end{proof}
Without loss of generality, assume that the collection $\cV$ given by Claim~\ref{claim:collection_r_blobs} is $\cV=\{V_1, \dots, V_r\}$. Consider the complete graph $\mathcal{R}$ on the vertex set $V_1, \dots, V_r$ and call an edge $V_i, V_j$ red if there are at least $\frac{f \xi^2 n^2 p}{16b (t+t')^2}$ red edges between $V_i$ and $V_j$ in $G$, and blue otherwise. We next show the following claim, which will later allow us to conclude that there is a large red clique in $\cR$.
\begin{figure}[!htbp]
        \centering
        \includegraphics[scale=0.6]{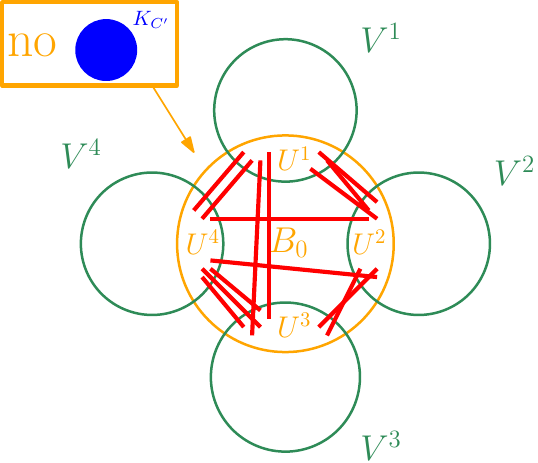}
        \caption{An illustration of the argument in Claim~\ref{claim:no_blue_K_b}, showing that the absence of a blue clique implies high density in red. Here we have $b=4$ and the sets $V^1,\dots, V^b$ are given in green. The intersections of these sets with $B_0$ (in orange) are $U^1, \dots, U^b$. Their union $U$ has no blue $K_{C'}$, implying by Tur\'an's theorem that there are many red edges in between $V^1, \dots, V^b$.}
        \label{fig:Turan_argument}
\end{figure}
\begin{claim}
\label{claim:no_blue_K_b}
There is no blue $K_b$ in $\mathcal{R}$.
\end{claim}
\begin{proof}
Consider some $b$-tuple of sets $V^1, \dots, V^b \in \mathcal{V}$ and some blue-avoiding block $B$ that intersects $\mathcal{V}$ nicely. Recall that $B_0$ contains no $K_{C'}$ in blue in $R$ and that $|B_0 \cap V^i| \geq \xi \frac{C}{t+t'}$ for each $i \in [b]$. Denote by $U^i$ a subset of size precisely $\xi \frac{C}{t+t'}$ of $B_0 \cap V^i$, for each $i \in [b]$. Letting $U := \cup_{i=1}^{b} U^i$, note that $U$ induces a clique in $G$ with no blue $K_{C'}$ in it. Then by Tur\'an's theorem, at least $\frac{1}{C'+1}\binom{|U|}{2}$ of the edges of $U$ are red. Note that $U$ has $\binom{\frac{b\xi C}{t+t'}}{2}$ edges, out of which $b\binom{\frac{\xi C}{t+t'}}{2}$ are internal, meaning that they are contained in some $U^i$. Thus at least $\frac{1}{C'+1}\binom{\frac{b \xi C}{t+t'}}{2} - b\binom{\frac{\xi C}{t+t'}}{2} \geq \frac{1}{4b}\binom{\frac{b \xi C}{t+t'}}{2} $ of the edges of $U$ are red and are not internal (see Figure~\ref{fig:Turan_argument} for an illustration of this argument).

Because there are at least $\frac{f n(n-1) p}{C(C-1)}$ blue-avoiding blocks $B$ that intersect $\cV$ nicely, it follows that there are at least $\frac{f n(n-1) p}{C(C-1)} \frac{1}{4b} \binom{\frac{b \xi C}{t+t'}}{2} \geq \frac{f b\xi^2 n^2 p}{16 (t+t')^2}$ red edges that are not internal in $V^1 \cup \dots \cup V^b$. Thus, at least one pair among $\{V^i\}_{i \in[b]}$, say $V^i$ and $V^j$, has at least $\frac{f \xi^2 n^2 p}{16b (t+t')^2}$ red edges between them. Therefore, the edge $V^i, V^j$ is red in $\cR$. Such an edge exists for any arbitrary $b$-tuple of sets in $\cV$, from which the claim follows.
\end{proof}

Since $r = r(21, b)$ and by Claim~\ref{claim:no_blue_K_b} there is no blue $K_b$ in $\mathcal{R}$ among $V_1, \dots, V_r$, there must be a red $K_{21}$, say given by $V_1, \dots, V_{21}$. Note that for each $i,j \in[21]$, $\{V_i, V_j\}$ is an $(\varepsilon, p)$-regular pair with at least $\frac{f \xi^2 n^2 p}{16b (t+t')^2}$ red edges. Thus the (red) density of this $(\varepsilon, p)$-regular pair is at least $\frac{f \xi^2 p}{32C'}$. Hence, $V_1, \dots, V_{21}$ are the desired sets, concluding the proof of the lemma.
\end{proof}
The next lemma allows us to conclude that if $21$ linear-sized sets of vertices are pairwise regular and dense enough in some colour in $G$, as given by the conclusion of Lemma~\ref{lem:density-from-cliques}, then the same holds for these $21$ sets in $\tilde{G}$.
\begin{lemma}
\label{lem:g-c-n-p-subsample-regular-sets}
Let $K,C \in \mathbb{N}$ with $C \geq 2$, $\mu>0$ and let $0<\gamma' \ll  \varepsilon \ll \gamma$. Let $G$ be an outcome of $G^C(n,p)$ which is $(\gamma',p)$-upper-uniform. Suppose there is a red/blue colouring of its edges, such that for disjoint sets of vertices $V_1, \dots, V_K$ with $|V_i| = \mu n$, the red subgraph of $G[V_i, V_j]$ is $(\varepsilon,p)$-regular with density at least $\gamma p$, for each $i \neq j$. Then w.h.p.\ $\tilde{G}$ is such that the red subgraph of $\tilde{G}[V_i, V_j]$ is $(4\varepsilon,\tilde{p})$-regular with density at least $\frac{\gamma}{2} \tilde{p}$, for each $i \neq j$.
\end{lemma}
\begin{proof}
Let $U, W$ be distinct sets among $V_1, \dots, V_K$. We show the statement of the lemma for $U$ and $W$, which then by a union bound holds for all such pairs. We refer to the red subgraphs of $G$ and $\tilde{G}$ as $G_r$ and $\tilde{G}_r$ respectively. First we show the following claim.
\begin{claim}
\label{concentration_of_density}
Let $\nu>0$ and let $U' \subseteq U$ and $W' \subseteq W$ with $|U'| \geq \varepsilon |U|$ and $|W'| \geq \varepsilon |W|$. Then w.h.p.\ it holds that  $d_{\tilde{G}_r}(U', W') = (1 \pm \nu) \frac{\tilde{p}}{p}d_{G_r}(U', W')$.
\end{claim}
\begin{proof}
Let $\cB$ be the collection of blocks from $S(2,C,n)$ which appear in $G$.
For each block $B$ in $\cB$, let $X_B$ be the random variable counting the red edges in $\tilde{G}$ between $U'$ and $W'$ contained in $B$. Note that $0 \leq X_B \leq \binom{C}{2}$. Denote by $X:= \displaystyle\sum_{B \in \cB} X_B$ the number of red edges between $U'$ and $W'$ in $\tilde{G}$. By linearity of expectation, and using that for each $B\in \cB$,
$$\mathbb{E}[X_B]= \sum_{e\in E(B) \cap E_{G_r}(U',W')} Pr\Big[e \in \tilde{G}\Big],$$
we have
$$\mathbb{E}[X] = \sum_{B \in \cB} \mathbb{E}[X_B] = e_{G_r}(U',W') Pr\Big[e \in \tilde{G} \text{ for a fixed } e \in E(B)\Big] = e_{G_r}(U',W') \frac{\tilde{p}}{p},$$
where the last equality follows by Lemma~\ref{edge_probability}. Therefore, $\mathbb{E}[X] = \Omega(n^2 \tilde{p})$, since $e_{G_r}(U',W') = \Omega(n^2 p)$ because $G_r[U,W]$ is $(\varepsilon, p)$-regular with density at least $\gamma p$. Applying Theorem~\ref{thm:mcdiarmid} with $t = \nu \mathbb{E}[X]$, we get that
$$ Pr\Big[|X - \mathbb{E}[X]| \geq \nu \mathbb{E}[X]\Big] \leq 2 e^{-\frac{2t^2}{|\cB| \binom{C}{2}^2}} = e^{-\Omega\Big(\frac{n^4\tilde{p}^2}{n^2}\Big)} = e^{ - \omega(n)}.
$$

The claim then follows from a union bound over all $2^{2n}$ possible choices of $U', W'$.
\end{proof}
By Claim~\ref{concentration_of_density}, we have that $d_{\tilde{G}_r}(U, W) \geq \gamma \tilde{p} / 2$. Let $U' \subseteq U$ and $W' \subseteq W$ be such that $|U'| \geq \varepsilon |U|$ and $|W'| \geq \varepsilon |W|$. By the triangle inequality and Claim~\ref{concentration_of_density} with $\nu:=\varepsilon$, we have
\begin{align*}
 & \ \left|d_{\tilde{G}_r}(U', W') - d_{\tilde{G}_r}(U, W)\right|\\ 
 \leq & \ \left|d_{\tilde{G}_r}(U', W') - \frac{\tilde{p}d_{G_r}(U', W')}{p}\right| + \left|\frac{\tilde{p}d_{G_r}(U', W')}{p} - \frac{\tilde{p}d_{G_r}(U, W)}{p}\right| + \left|\frac{\tilde{p}d_{G_r}(U, W)}{p} - d_{\tilde{G}_r}(U, W)\right| \\
 \leq &\ \varepsilon \frac{\tilde{p}d_{G_r}(U', W')}{p} + \varepsilon \tilde p + \varepsilon \frac{\tilde{p}d_{G_r}(U, W)}{p} \leq 2 \varepsilon(1+\gamma') \tilde{p} + \varepsilon \tilde{p} \leq 4\varepsilon \tilde{p}, 
\end{align*}
where we used that $G$, and therefore also $G_r$, is $(\gamma',p)$-upper-uniform (for $\gamma' \ll \varepsilon$) in the penultimate inequality, and that $G_r[U, W]$ is $(\varepsilon, p)$-regular for bounding the middle term in the second line.
\end{proof}

\subsection{The complete bipartite graphs}
\label{subsec:bipartite}
We begin the section by showing that w.h.p.\ the blocks in $G$ can be partitioned into block matchings, defined below. 

\begin{defin}[A block matching in $G$]
Let $\cS$ be the set of blocks in $G$. A subset $M \subseteq \cS$ is \emph{a block matching} if all blocks in $M$ are pairwise vertex-disjoint.
\end{defin}

We will also want each of the matchings to cover almost all the vertices, i.e. each matching will cover all but $\eta n$ vertices, where $\eta>0$ is small enough and we specify it later. Essentially, in the proof we will take $\eta$ to be the smallest of all constants we use.
\begin{defin}[Collection of almost perfect block matchings]
\label{defin:almost-covering-matchings}
 Let $\mathcal{S}$ be the set of blocks in $G$. A \emph{collection of almost perfect block matchings} is a family $\mathcal{M}$ of pairwise disjoint block matchings $M_1, \dots, M_z$ with $\frac{(1-\eta)(n-1)p}{C-1} \leq z \leq \frac{(1+\eta)(n-1)p}{C-1}$, such that each $M_i$ covers all but at most $\eta n$ vertices.
\end{defin}
Note that since w.h.p.\ $G$ contains at most $\frac{(1+\eta)n(n-1)p}{C(C-1)}$ many blocks, every collection of almost perfect block matchings is such that the number of blocks not appearing in any block matching in it is w.h.p.\ at most $ \frac{3\eta n(n-1)p}{C(C-1)}$.
\begin{lemma}
\label{lem:partition-into-matchings}
With high probability, there exists a collection of almost perfect block matchings $\mathcal{M}$ of the blocks in $G$.
\end{lemma}
\begin{proof}
 Consider the hypergraph $H$ with $V(H) = V(G)$ where the set of hyperedges $\mathcal{S}$ is the set of all blocks in $G$. We will show the existence of $z$ disjoint matchings in $H$ that each cover all but at most $\eta n$ vertices, where $z$ is as in Definition~\ref{defin:almost-covering-matchings}. We apply Lemma~\ref{lem:chromatic_index} to $H$ with $r:=C$, $\gamma:=\eta^3$, $\beta \ll \gamma, C^{-1}$, and $t:=\frac{(n-1)p}{C-1} > \sqrt{n}$. We can do this since for each vertex $v$ in $H$, we have $\mathbb{E}[d_H(v)] = \frac{(n-1)p}{C-1}$ and by using a Chernoff bound and a union bound over all vertices, we have w.h.p.\ that $(1-\beta/2)t \leq d_H(v) \leq (1+\beta/2)t$, and for every pair of vertices $u \neq v$ in $H,$ their codegree satisfies $d_H(u,v) \leq 1 < \beta t$, since every pair of vertices is in at most one block in $S(2,C,n)$. Therefore, there is a partition of the hyperedges of $H$ into matchings $M'_1, \dots, M'_{z'}$ for some $z' \leq (1+\gamma)t = \frac{(1+\gamma)(n-1)p}{C-1}$.

Suppose for contradiction fewer than $(1-\eta)t$ of the matchings $M'_1,\dots,M'_{z'}$ cover each at least $(1-\eta)n$ of the vertices of $H$. Since each vertex has degree in $H$ at least $(1-\beta/2)t>(1-\gamma)t$, by double counting pairs $(v,M'_i)$ such that $M'_i$ covers $v$, we get
$$ (1-\gamma)tn \leq (1-\eta)tn + (1+\gamma-(1-\eta))t(1-\eta)n ,$$
which implies $2\gamma \geq \gamma\eta + \eta^2$, contradicting our choice of $\gamma$. To finish, we choose the hypergraph matchings that each cover at least $(1-\eta)n$ of the vertices to be the required sets $M_1, \dots, M_z$.
\end{proof}

Since our final random graph construction builds on top of $G^C(n,p)$, we will now fix an outcome of $G^C(n,p)$ which has some useful properties which hold w.h.p. in that random graph model.
From now on, for each $C$ and all large enough $n$, we assume $G$ is an outcome of $G^C(n,p)$ for which the conclusions of Lemmas~\ref{lem:upper-bound-edges-of-G},~\ref{lem:density-from-cliques} and~\ref{lem:partition-into-matchings} hold, and additionally the following holds\footnote{\label{footnote:constants-of-lemmas}Note that here we did not yet specify the explicit constants which we use in those lemmas, but observe that any choice of constants which satisfies the relations in the lemmas works when $n$ is large enough. We specify the constants later in the proof, when we explicitly call the mentioned lemmas.}: if we now take $\tilde{G}$ to be subsampled from $G$ as in Defition~\ref{defin:subsample-G-C-n-p}, then with probability at least $0.9$, the conclusions of both Lemma~\ref{lem:typical_vertices} and Theorem~\ref{thm:embed-in-regular-pairs} hold for $\tilde{G}$, and $\tilde{G}$ is $(\zeta,\tilde{p})$-upper-uniform for all constant $\zeta>0$. Note that such a $G$ exists since the conclusions of Lemmas~\ref{lem:upper-bound-edges-of-G},~\ref{lem:density-from-cliques} and~\ref{lem:partition-into-matchings} hold w.h.p.\ for $G^C(n,p)$, and the conclusions of Lemma~\ref{lem:typical_vertices} and Theorem~\ref{thm:embed-in-regular-pairs}, as well as $(\zeta,\tilde{p})$-upper-uniformity hold w.h.p.\ for $G(n,\tilde{p})$. Recall that by Lemma~\ref{lem:g-c-n-p-subsample-g-n-p} the outcome $\tilde{G}$ of the two-step process of sampling $G\sim G^C(n,p)$ and then subsampling it to get $\tilde{G}$ is distributed as $G(n,\tilde{p})$. Fubini's theorem thus implies that with probability at least $1/2$, the outcome of the first process is a graph $G$ for which the following is true. If we subsample from $G$ to get $\tilde{G}$, then with probability at least $0.9$ we have that $\tilde{G}$ is $(\zeta,\tilde{p})$-upper-uniform for every constant $\zeta>0$ and is such that Lemma~\ref{lem:typical_vertices} and Theorem~\ref{thm:embed-in-regular-pairs} hold. 

\begin{defin}[The host graph]
\label{defin:the-a-i-s}
Let $\mathcal{M}=(M_1, \dots, M_z)$ be a collection of almost perfect block matchings for the graph $G$ as in Definition~\ref{defin:almost-covering-matchings}. For each $i \in [z]$, let $n_i := |M_i|$ and $G_i \sim G(n_i,p')$ with probability\footnote{Here we could have chosen any function $\alpha(n)$ growing to infinity with $n$ instead of $\log n$, and our arguments still would go through.} $p' = \frac{\log{n}}{n}$, where each vertex in $G_i$ corresponds to a block in $M_i$. Let $A_1, \dots, A_z$ be the collection of random graphs defined as $A_i = G_i \boxtimes I_C$, that is, blow-ups of $G_i$ by an independent set of size $C$. Similarly, let $A'_1, \dots, A'_z$ be defined as $A'_i = G^3_i \boxtimes I_C$. We identify the sets $I_C$ in each $A_i$ and $A_i'$ with the corresponding blocks in $M_i$, thus defining each $A_i$ and each $A'_i$ on the vertex set of $G$ (see Figure~\ref{fig:host}). Now we are ready to define our \emph{host graph}, namely it is the union of graphs $\Gamma = G\cup A_1' \cup \ldots \cup A_z'$.
\end{defin}

\begin{figure}[!htbp]
        \centering
        \includegraphics[scale=0.4]{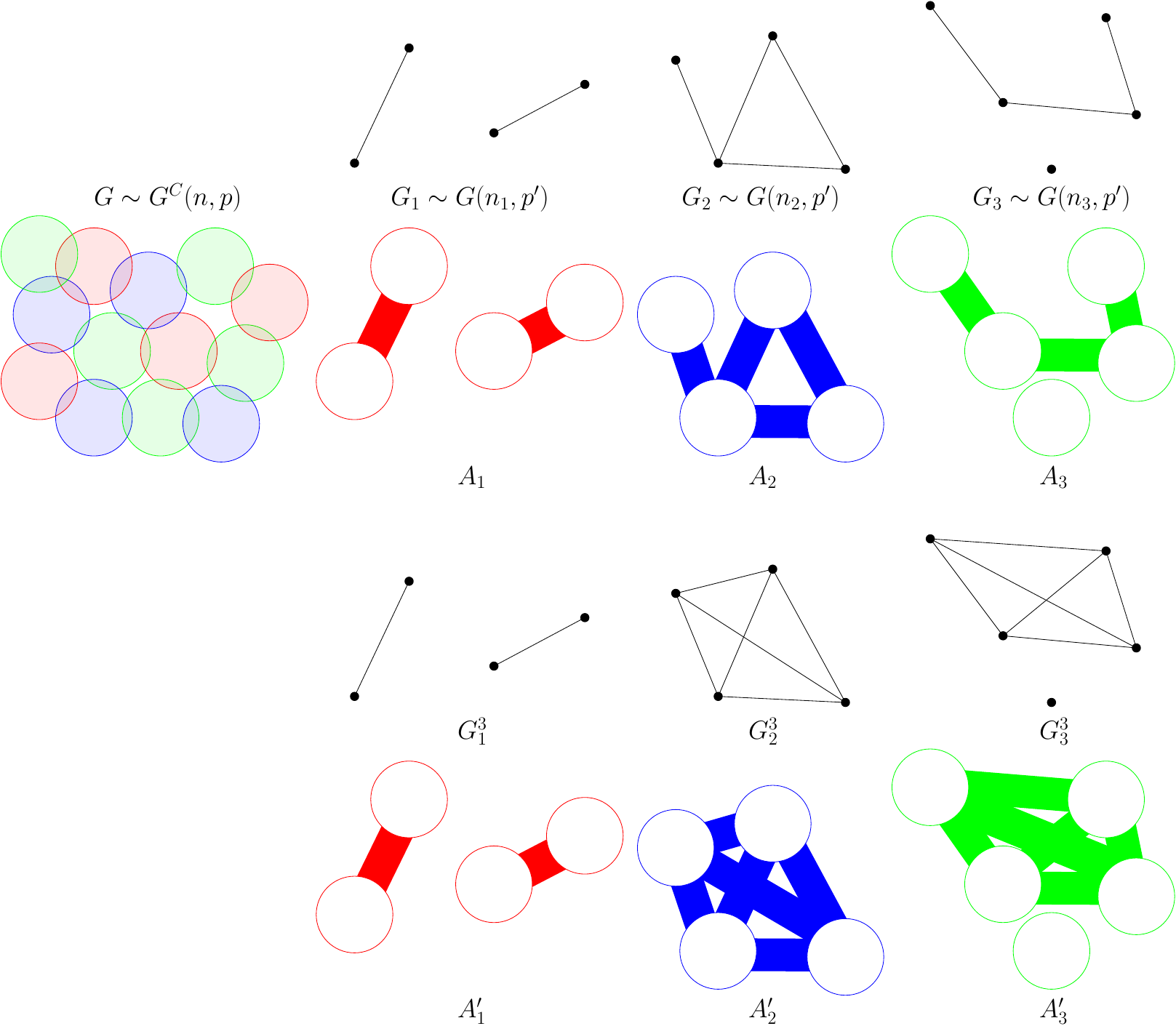}
        \caption{Obtaining the host graph: $G$ contains disjoint almost perfect block matchings $M_1, M_2, M_3$, each of which covers almost all vertices. Each graph $A_i$ is obtained from a binomial random graph $G_i$ by replacing each vertex with a copy of $I_C$, and identifying these $I_C$'s with the blocks in $M_i$. Each $A'_i$ is obtained in the same way, but from $G^3_i$ instead of $G_i$.}
        \label{fig:host}
\end{figure}

\begin{lemma}
\label{lem:host-graph-edges-upper-bound}
The host graph $\Gamma$ w.h.p.\ has at most $n^{\frac{3}{2} + 2\delta}$ edges.
\end{lemma}
\begin{proof}
We have that $E(\Gamma) = E(G) \cup E(A'_1) \cup \ldots \cup E(A'_z)$. By Lemma~\ref{lem:upper-bound-edges-of-G}, we have $e(G) \leq n^{\frac{3}{2} + \delta}$.

For each $i\in [z]$, we have $n_i \leq \frac{n}{C}$ since $M_i$ is a block matching. Thus, by the Chernoff bound and the union bound, w.h.p.\ $\Delta(G_i) \leq 2C n_i \frac{\log{n}}{n} \leq 2 \log{n}$, so w.h.p.\ $\Delta(G^3_i) \leq 8 \log^3{n}$. Therefore, $e(G^3_i) \leq 8 n \log^3{n}$ and $e(A'_i) \leq 8C^2 n \log^3{n}$. Since $z \leq \frac{2np}{C}$ by Lemma~\ref{lem:partition-into-matchings}, we have $e(A'_1) + \dots + e(A'_z) \leq n^2 p \log^4{n} \leq n^{\frac{3}{2} + 2\delta}/2$.
\end{proof}

In a manner similar to Definition~\ref{defin:subsample-G-C-n-p}, we can subsample from the graphs $A_i$ to get subgraphs of binomial random graphs.
\begin{defin}
\label{defin:subsample-a-is}
We define the collection of random graphs $\tilde{A}_1, \dots, \tilde{A}_z$ in the following way.
For each $i\in [z]$ and for each copy $D$ of $K_{C,C}$ in $A_i$ that corresponds to an edge in $G_i$, we sample a non-empty subset of the edges $D' \subseteq E(D)$ to be present in $\tilde{A}_i$ with the following probability
$$ Pr\Big[E(D) \cap E(\tilde{A}_i)=D'\Big] = \frac{ \tilde{p}'^{|D'|} (1-\tilde{p}')^{C^2 - |D'|} }{p'},$$
where $\tilde{p}'$ is given by $p' = 1-(1-\tilde{p}')^{C^2}$, that is, $\tilde{p}' \approx \frac{p'}{C^2}=\frac{\log n}{nC^2}$. 
\end{defin}
In the remainder of the paper,
we additionally use the probabilities
$p'' = 1-(1-p')^z \approx \frac{p \log{n}}{C}$ and $\tilde{p}'' = 1-(1-\tilde{p}')^z \approx \frac{p\log{n}}{C^3}$, whose meaning will become 
apparent later in this section. For clarity of presentation, we provide a table with all edge probabilities we use, along with their definitions and asymptotic behaviour.
\renewcommand{\arraystretch}{2}
\begin{table}[H]
    \centering
    \begin{tabular}{|c|c|c|c|c|c|c|}
        \hline
        Notation & $p$ & $\tilde{p}$ & $p'$ & $\tilde{p}'$ & $p''$ & $\tilde{p}''$ \\
        \hline
        Definition & $n^{\delta - 1/2}$ & $p=1 {-} (1{-}\tilde{p})^{\binom{C}{2}}$ & $\frac{\log{n}}{n}$ & $p' = 1{-}(1{-}\tilde{p}')^{C^2}$  & $1{-}(1{-}p')^z$ & $1{-}(1{-}\tilde{p}')^z$ \\
        \hline
        Asymptotics & $n^{\delta - 1/2}$ & $\frac{n^{\delta - 1/2}}{\binom{C}{2}}$ & $\frac{\log{n}}{n}$ & $\frac{\log{n}}{nC^2}$ & $\frac{n^{\delta - 1/2}\log{n}}{C}$ & $\frac{n^{\delta - 1/2}\log{n}}{C^3}$ \\
        \hline
    \end{tabular}
    \caption{Probabilities}
    \label{tab:probabilities}
\end{table}

\begin{lemma}
\label{edge_probability_a_i}
For $i \in [z]$, let $A_i = G_i \boxtimes I_C$ for any outcome of $G_i \sim G(n_i, p')$. Then for any copy $D$ of $K_{C,C}$ whose edges are present in $A_i$, and any $e \in E(D)$, we have
$Pr[e \in E(\tilde{A}_i)] = \frac{\tilde{p}'}{p'}\approx \frac{1}{C^2}$.
\end{lemma}
\begin{proof}
We have
\begin{align*}
Pr\Big[e \in E(\tilde{A}_i)\Big] &= \sum_{D' \subseteq E(D), e\in D'} Pr\Big[E(D) \cap E(\tilde{A}_i) = D'\Big] \\
&=
\sum_{s=1}^{C^2} \binom{ C^2-1 }{s-1} \frac{\tilde{p}'^s (1-\tilde{p}')^{C^2 - s}}{p'} =\frac{\tilde{p}'}{p'} \sum_{s=0}^{C^2-1} \binom{ C^2-1 }{s} \tilde{p}'^s (1-\tilde{p}')^{C^2 - 1 - s} = \frac{\tilde{p}'}{p'}.\\
\end{align*}
\end{proof}

The next lemma considers the two-step process of first sampling $A_1, \dots, A_z$, and then subsampling it to get  $\tilde{A}_1, \dots, \tilde{A}_z$, and shows that $\tilde{A}_1 \cup \dots \cup \tilde{A}_z$ behaves as a subgraph of a binomial random graph.
\begin{lemma}
\label{lem:a-i-subgraph-of-g-n-p}
 The graph $\tilde{A}_1 \cup \dots \cup \tilde{A}_z$ can be viewed as a subgraph of $G(n,\tilde{p}'')$, where $\tilde{p}'' = 1-(1-\tilde{p}')^z \approx z \tilde{p}' \approx \frac{p\log n}{C^3}$.
\end{lemma}
\begin{proof}
We couple the sampling process of $\mathcal{A} := \tilde{A}_1 \cup \dots \cup \tilde{A}_z$ with that of $F \sim G(n,\tilde{p}'')$ in such a way that $\mathcal{A} \subseteq F$. We sample $F$ using multiple exposure by first sampling  $F_i \sim G(n,\tilde{p}')$ independently for each $i\in[z]$, and then taking $F := F_1 \cup \dots \cup F_z$. We then define $L_i$ in terms of $F_i$ in the following way. For each $\{u,v\} \in F_i$, add $\{u,v\}$ to $L_i$ if and only if there are two distinct blocks $B, B' \in M_i$ such that $u \in B$ and $v \in B'$. Take $L := L_1 \cup \dots \cup L_z $.

Since $L_i \subseteq F_i$ and therefore $L = L_1 \cup \dots \cup L_z \subseteq F_1 \cup \dots \cup F_z = F$, it is enough to show that $L_i$ is indeed distributed as $\tilde{A}_i$. To do this, consider some $E' \subseteq E(K_n)$. Firstly, if some $\{u,v\} \in E'$ is such that there are no $B,B' \in M_i$ with $u \in B$ and $v \in B'$, then $Pr[E(L_i) = E'] = Pr[E(\tilde{A}_i)=E']= 0$. Next, since in both $\tilde{A}_i$ and $L_i$ the edges between different pairs of blocks are sampled independently, it is sufficient to show that for each complete bipartite graph $D$ between two blocks $B, B' \in M_i$ and for each $D' \subseteq E(D)$,
$$ Pr\Big[E(\tilde{A}_i) \cap E(D) = D'\Big] = Pr\Big[E(L_i) \cap E(D) = D'\Big].$$
We have that the edges of $L_i \cap D$ behave precisely as in a random graph with edge probability $\tilde{p}'$, so
$$ Pr\Big[E(L_i) \cap E(D) = D'\Big] = \tilde{p}'^{|D'|} (1- \tilde{p}')^{C^2 - |D'|}.$$
On the other hand, for $\tilde{A}_i \cap D$ we analyse two cases depending on whether $D'$ contains at least one edge or not. If $D' = \varnothing$, then, using that $1-p'=(1-\tilde{p}')^{C^2}$,
$$ Pr\Big[E(\tilde{A}_i) \cap E(D) = \varnothing\Big] = Pr\Big[E(A_i) \cap E(D) = \varnothing\Big] = (1-\tilde{p}')^{C^2} =Pr\Big[E(L_i) \cap E(D) = \varnothing\Big].$$
If $D' \neq \varnothing$, then
$$ Pr\Big[E(\tilde{A}_i) \cap E(D) = D'\Big] = Pr\Big[ E(\tilde{A}_i) \cap E(D) = D' | E(D) \subseteq E(A_i)\Big]\cdot Pr\Big[E(D) \subseteq E(A_i)\Big]$$
$$ = \frac{ \tilde{p}'^{|D'|} (1-\tilde{p}')^{C^2 - |D'|} }{p'} \cdot p' = \tilde{p}'^{|D'|} (1-\tilde{p}')^{C^2 - |D'|} = Pr\Big[E(L_i) \cap E(D) = D'\Big].$$
\end{proof}

\begin{lemma}
\label{lem:upper-uniformity-ais}
For every $0 < \gamma <1$, w.h.p.\ the graph $A_1 \cup \dots \cup A_z$ is $(\gamma,p'')$-upper-uniform, where $p'' = 1 - (1-p')^z \approx zp' \approx \frac{p\log{n}}{C} $.
\end{lemma}
\begin{proof}
Consider some $U$ and $W$ with $|U|, |W| \geq \gamma n$ and $U 
\cap W = \varnothing$. Fix some $i \in [z]$. Let $X_i$ be the random variable which counts the edges between $U$ and $W$ in $A_i$. For each copy $D$ of $K_{C,C}$ whose two parts are two blocks from $M_i$ and which has at least one edge between $U$ and $W$, denote by $X_D$ the number of edges from $D$ contained in $A_i$ with one endpoint in $U$ and the other endpoint in $W$. Note that $X_i=\sum_D X_D$. Observe that the variables $X_D$ are independent and take values between $0$ and $C^2$. Since we are showing an upper bound, we can w.l.o.g.\footnote{without loss of generality} assume that each pair of vertices in $(U,W)$ is covered by some $D$, as otherwise we can add random variables $X_D$ to the sum where each new $D$ is the bipartite graph corresponding to just one uncovered pair of vertices. By Lemma~\ref{lem:weighted-Chernoff} applied to $X_i$, we get
$$ Pr\Big[X_i > \mathbb{E}[X_i] + \frac{\gamma}{2} |U||W| p' \Big] \leq e^{-\Theta(\mathbb{E}[X_i])} \leq e^{-\Theta(n\log{n})}.$$
It follows from a union bound over all $2^{2n}$ possible choices of $U,W$ and $z=\Theta(np)$ choices of $i\in[z]$ that w.h.p.\ for all $i\in [z]$, the upper bound $X_i \leq \mathbb{E}[X_i] + \frac{\gamma}{2} |U||W|p'$ holds. Thus, since $$\mathbb{E}[X_i] =\sum_{u\in U, w\in W} Pr\Big[uw \in E(A_i)\Big] \leq |U||W|p',$$
we have that
$$ e_{A_1 \cup \dots \cup A_z}(U,W) \leq \sum_{i=1}^z X_i \leq \sum_{i=1}^z \Big(\mathbb{E}[X_i] + \frac{\gamma}{2} |U||W|p'\Big) \leq \Big(1+\frac{\gamma}{2}\Big)|U||W|p'z.$$
On the other hand, $p'' = 1 - (1-p')^z \geq zp' - z^2 (p')^2 = zp' - o(zp')$, so
$$e_{A_1 \cup \dots \cup A_z}(U,W) \leq (1+\gamma) |U||W|p''.$$
\end{proof}

The next technical lemma shows that the same edge is never in many $A_i$'s.
\begin{lemma}
\label{lem:not-too-many-edges-in-more-than-one-ai}
With high probability, there are at most $n^{\frac{3}{2}}$ edges that occur in more than one $A_i$ and there are no edges that occur in at least $5$ of them.
\end{lemma}
\begin{proof}
We have that for any two vertices $u,w$, $Pr[uw \in E(A_i)] \leq p'$, so
$$ Pr\Big[uw \text{ is in at least } k \text{ of the }A_i\text{'s}\Big] \leq \binom{z}{k} (p')^k \leq (zp')^k = O\big(n^{(2\delta - \frac{1}{2})k}\big).$$
Let $X_k$ be the number of edges that are in at least $k$ of the $A_i$'s. Then
$ \mathbb{E}[X_k] = O(n^{2-\frac{k}{2} + 2\delta k}).$ 
Setting $k=5$, we get $\mathbb{E}[X_5] = O(n^{10\delta - 0.5}) = o(1)$, so by the first moment method w.h.p.\ $X_5 = 0$.

For $k=2$, we apply McDiarmid's inequality (Theorem~\ref{thm:mcdiarmid}) with  $t:= n^{3/2}/2$, where the coordinates that influence $X_2$ are all the $O(n^2z) = O(n^{2.5+\delta})$ potential complete bipartite graphs in the $A_i$'s (which are sampled independently). Note that each coordinate's effect is at most $C^2$. Then, since $\mathbb{E}[X_2]\leq n^{3/2}/2$
$$ Pr\Big[X_2 >n^{\frac{3}{2}}\Big] \leq Pr\Big[X_2-\mathbb{E}[X_2] \geq t\Big] \leq e^{-\frac{2t^2}{O(n^{2.5+\delta})}} = e^{-\Omega(n^{0.5-\delta})}.$$
\end{proof}

We now introduce the concept of a \emph{densifier} of some $A_i$, which is a structure that guarantees some density in one of the colours in $A_i$. The next lemma shows how to infer from the existence of sufficiently many densifiers in say red, that there are $21$ linear-sized sets in $\Gamma$, all pairs of which are regular and dense in red. This is a key ingredient in the proof of Theorem~\ref{thm:main} which shows that the induced cycles from the decomposition of $H$ can be embedded in the same colour as the bounded treewidth part.
\begin{defin}[A coloured $(C',\gamma, s, q)$-densifier]
\label{def:densifier}
Let $C',s,q$ be integers with $s \leq C'\leq C$, and let $\gamma>0$.
Let $A_i$ be coloured in red and blue, and $S\subseteq V(A_i)$. Denote by $\mathcal{I} = \{I_1, I_2, \ldots, I_{n_i} \}$ the copies of the independent set $I_C$ on $C$ vertices in $A_i$ that correspond to the blocks in $M_i$. Then a \emph{red $(C', \gamma, s, q)$-densifier} of $A_i[S]$ consists of $q$ disjoint families $W_1,\dots,W_q$ each containing $\gamma \frac{n}{C'}$ independent sets, so that each $I \in W_k$ is such that $I \subset I_j \cap S$ for some $j\in [n_i]$ with $|I| = C'$, and all the $I$'s are vertex disjoint across all the $W_k$'s. Furthermore, for every pair $I\in W_k,I' \in W_{k'}$ with $k \neq k'$, there is no blue $K_{s,s}$ between $I$ and $I'$ in $A_i$ (see Figure~\ref{fig:densifier}).
\end{defin}

\begin{figure}[!htbp]
        \centering
        \includegraphics[scale=0.4]{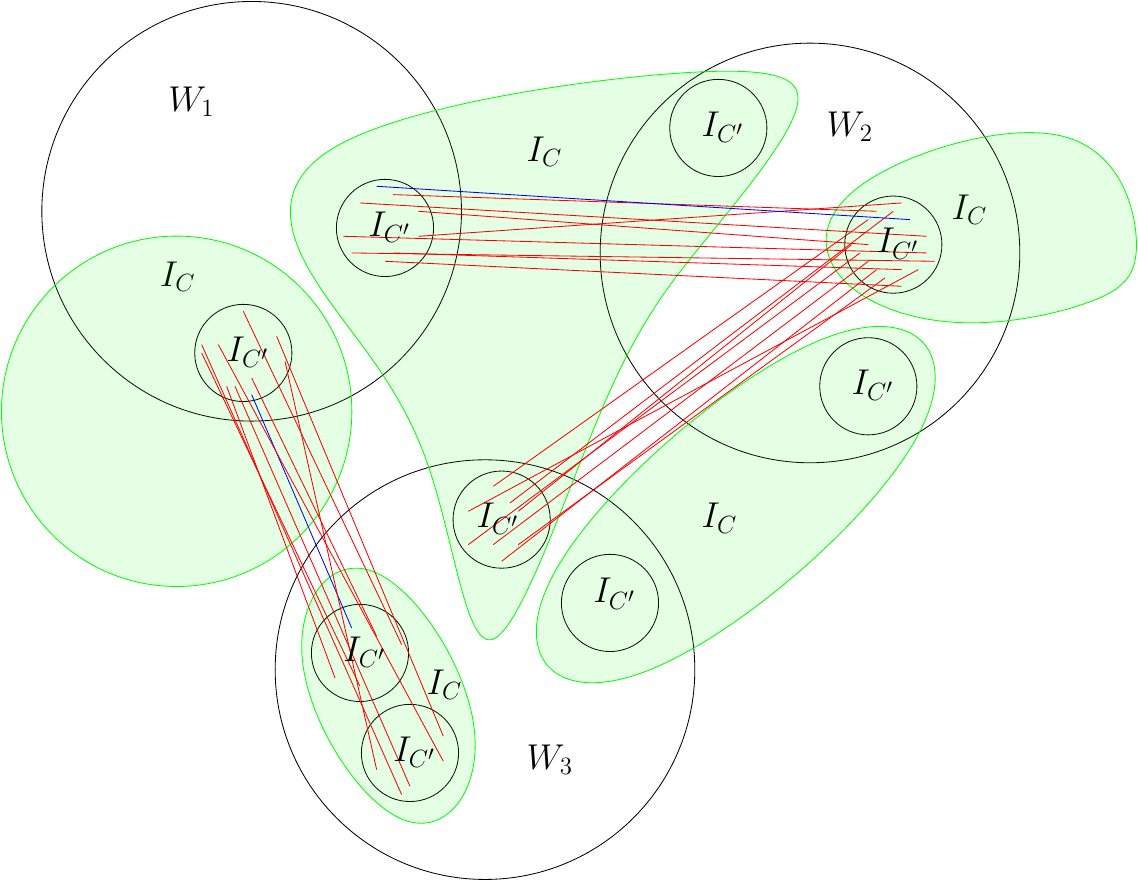}
        \caption{A red densifier: Each family $W_k$ consists of a number of copies of $I_{C'}$, each of which is a subset of some copy of $I_C$ in $A_i$. Whenever two copies of $I_C$ are connected by a complete bipartite graph in $A_i$, each subgraph of that graph induced between two copies of $I_{C'}$ in different families $W_k$ and $W_{k'}$ contains no blue copy of $K_{s,s}$ and so most of its edges are red.}
        \label{fig:densifier}
\end{figure}

\begin{lemma}
\label{lem:regular-clique-complete-bipartite-graphs}
Let $q \geq 1000$ and $s \in \mathbb{N}$. For any $\mu \ll \varepsilon \ll \gamma \ll  \{ \alpha,q^{-1}\}$ and $C\gg C' \gg \{s, \mu^{-1}\}$, w.h.p.\ the following holds. Let $R \subset V(\Gamma)$ be a set of size $|R| = \alpha n$ and consider a colouring of $\Gamma$ such that for at least half of the $A_i$'s there exists a red $(C',\gamma, s, q)$-densifier of $A_i[R]$. Then there are disjoint subsets of vertices $V_1, \dots, V_{21} \subset R$, each of size $m \geq \mu n$, such that each pair $(V_i, V_j)$ is $(\varepsilon, p'')$-regular in the red subgraph of $A_1 \cup \dots \cup A_z$ with density at least $\tau p''$ for $\tau = \tau(\gamma, \alpha, q)$.
\end{lemma}
\begin{proof}
Let $A := A_1 \cup \dots \cup A_z$. We first give an overview of the proof. We start by applying the sparse regularity lemma (Theorem~\ref{thm:sparse-regularity-lemma}) to the red subgraph of $A[R]$. We get sets $V_1, \dots, V_t$, most pairs of which are regular. Next, we show by double-counting in Claim~\ref{claim:intersectionVW} that for each $i \in [z]$, and for each part $W$ of the densifier of $A_i$ (if such exists), at least some fraction of the sets $V_k$ with $k\in[t]$ have a relatively large intersection with $W$. We then show with the help of Claim~\ref{claim:intersectionVW} there exist $21$ sets $V^1, \dots, V^{21}$ among $V_1, \dots, V_t$ such that all pairs $V^a, V^b$ are regular in red and at least a fraction of the $A_i$'s respective densifiers each contain sets $W^1, \dots, W^{21}$ such that for each $j$, $V^j$ and $W^j$ have a large intersection (Claim~\ref{claim:nice-pairs}). This kind of `alignment' between the $V^j$'s and the $W^j$'s then allows us to conclude that all pairs $V^a, V^b$ also have high density in red. The reason is that by Claim~\ref{claim:intersectionVI}, many independent sets $I^a \in W^a$ and $I^b \in W^b$ each have a large intersection with $V^a$ and $V^b$ respectively. Since there is no blue $K_{s,s}$ between $I^a$ and $I^b$ by the definition of a densifier, there must be many edges in red whenever a complete bipartite graph is present between $I^a$ and $I^b$ (Theorem~\ref{thm:kovari-sos-turan}). A concentration inequality lets us conclude that many such complete bipartite graphs exist in $A_i$, giving us the required density in red.

Let $\tau := \Big(\frac{q\gamma^2}{800\alpha^2 e}\Big)^{21}\frac{ \gamma^2 }{10^5 \alpha^2}$. 
Let $t_0 \gg 1$ and let $T= T_{\ref{thm:sparse-regularity-lemma}}(t_0,\varepsilon)$ be the upper bound on the number of sets given by the sparse regularity lemma.

By Lemma~\ref{lem:upper-uniformity-ais}, for all $ \gamma'>0$ the graph $A$ is w.h.p.\ $(\gamma',p'')$-upper-uniform, which implies that the red subgraph of $A[R]$ is also $(\gamma',p'')$-upper-uniform for every fixed $\gamma' > 0$. We apply Theorem~\ref{thm:sparse-regularity-lemma} to the red subgraph of $A[R]$ with $\varepsilon$ and $t_0$ and get an equipartition $V_1, \dots, V_t$ of $R$ with $t_0 \leq t \leq T$ such that all but $\varepsilon \binom{t}{2}$ pairs $V_i, V_j$ are $(\varepsilon,p'')$-regular in red, and such that each $V_i$ has size $|V_i|=m$.

Without loss of generality, assume that $A_1, \dots, A_{z/2}$ have an associated red densifier.
Let $\cW_i = \{W_1,$ $\dots,$ $ W_q\}$ denote the $(C',\gamma,s,q)$-densifier of $A_i$ with $i\in[z/2]$ and let $V(W_j)$ be the vertices in independent sets in $W_j$. We now show several simple counting claims.
\begin{claim}
\label{claim:intersectionVW}
Let $\beta = \frac{\gamma }{3\alpha}$, suppose $i \in[z/2]$, and let $W_j \in \cW_i$. Then at least $\beta t$ of the $V_k$'s have the property that $|V_k \cap V(W_j)| \geq \beta |V_k|$.
\end{claim}
\begin{proof}
Suppose this is not the case. Note that $|W_j| = \frac{\gamma n}{C'}$, so the total number of vertices contained in sets from $W_j$ is equal to $v(W_j):= |V(W_j)| = \gamma n$. Then
$$ \gamma n \leq  \beta t m + (1-\beta) t \beta m,$$
where the right hand side is an upper bound on the number of vertices in $W_j$, since the first term bounds the vertices in sets $V_k$ with intersection with $W_j$ at least $\beta m$, and the second term---all the others. This implies
$$ \gamma \leq \beta \alpha + (1-\beta)\beta \alpha \leq 2 \beta \alpha,$$
where we used $mt=\alpha n$, thus contradicting our choice of $\beta$.
\end{proof}

\begin{claim}
\label{claim:intersectionVI}
Suppose $i \in [z/2]$, let $W_j \in \cW_i$ and let $V_k$ be such that $|V(W_j) \cap V_k| \geq \beta |V_k|=\beta m$. Then at least $\beta m/2$ of the vertices of $V_k$ each belong to some $I \in W_j$ such that $|I \cap V_k| \geq \frac{C'}{12t}$.
\end{claim}
\begin{proof}
Suppose for contradiction that this is not the case. Then at least $\frac{\beta m}{2}$ of the vertices of $V_k$ belong to some $I \in W_j$ such that $|I \cap V_k| < \frac{C'}{12t}$. To cover those $\beta m/2$ vertices of $V_k$, the number of such $I \in W_j$ that are required is at least
$$ \frac{\beta m / 2}{C'/(12t)} = \frac{6 \beta \alpha n }{C'} = \frac{2 \gamma n}{C'} > |W_j|,$$
contradicting the number of available $I \in W_j$.
\end{proof}

For any 21 sets $V^1, \dots, V^{21}$ from the regularity equipartition, and $i\in [z/2]$, we say that the pair $\big(\{V^1, \dots, V^{21}\}, A_i\big)$ is \emph{nice} if all pairs $V^j, V^k$ are regular in the red subgraph $A^r$ of $A$ and there are distinct sets $W^1, \dots, W^{21} \in \cW_i$, such that for every $j \in [21]$ we have $|V^j \cap V(W^j)| \geq \beta m$. Note that the number of \emph{irregular} $21$-tuples $V^1, \dots, V^{21}$ (i.e. the tuples for which there is at least one pair $V^j, V^k$ that is not regular in $A^r$) is at most $\varepsilon t^{21}$.
\begin{claim}
\label{claim:nice-pairs}
For $\lambda := \Big(\frac{q\gamma^2}{800\alpha^2 e}\Big)^{21}$, there is a $21$-tuple $\{V^1, \dots, V^{21}\}$ such that for at least $\lambda z/2$ of the $A_i$'s, $(\{V^1, \dots, V^{21}\}, A_i)$ is nice.
\end{claim}
\begin{proof}
Suppose for contradiction there is no such $21$-tuple. Since there are $\binom{t}{21}$ tuples in total, the number of nice pairs is then at most $\binom{t}{{21}} \lambda z/2$. Let us now show a lower bound on the number of nice pairs $\big(\{V^1, \dots, V^{21}\}, A_i\big)$.

First, note that in each of the $z/2$ considered $A_i$'s, there are $\binom{q}{21}$ tuples $W^1, \dots, W^{21} \in \cW_i$. Fix such a tuple $W^1, \dots, W^{21}$. By Claim~\ref{claim:intersectionVW}, for each $W^j$ there are at least $\beta t$ of the $V_k$'s with $|V_k \cap V(W_j)| \geq \beta m$. Thus, there are at least $\binom{\beta t}{21}$ tuples  $V^1, \dots, V^{21}$, such that $|V(W^j) \cap V^j| \geq \beta m$ for each $j\in[21]$. Among these, at most $\varepsilon t^{21}$ tuples are irregular in $A^r$. This gives rise to at least $\frac{z}{2} \binom{q}{21} \big(\binom{\beta t}{21} - \varepsilon t^{21}\big)$ nice pairs, but note that for each $i \in [z/2]$, we have potentially counted each tuple $V^1, \dots, V^{21}$ multiple times. Namely, each tuple $V^1, \dots, V^{21}$ which forms a nice pair with $A_i$ is counted at most $1 / \beta^{21}$ many times since each $V^k$ can have an intersection of size at least $\beta m$ with at most $1/\beta$ many $W_j$'s in $\cW_i$. Thus, there are at least $\frac{z}{2} \binom{q}{21} \big(\binom{\beta t}{21} - \varepsilon t^{21}\big) \beta^{21}$ many nice pairs. Comparing this to the upper bound from above, we get
$$  \frac{z}{2} \binom{q}{21} \Bigg(\binom{\beta t}{21} - \varepsilon t^{21}\Bigg) \beta^{21}  \leq \binom{t}{{21}} \lambda z/2.$$
Thus, since $\beta \gg \varepsilon$, this implies
$ \frac{q^{21}}{2 \cdot 21^{42}} t^{21} \beta^{42} \leq \frac{\lambda e^{21} t^{21}}{21^{21}} $, which boils down to $ \Big(\frac{q\gamma^2}{9\alpha^2 21 e}\Big)^{21} \leq 2\lambda,$
contradicting our choice of $\lambda$.
\end{proof}

Pick $\{V^1, \dots, V^{21}\}$ such that for at least $\lambda z/2$ of the $A_i$'s, the pair $(\{V^1, \dots, V^{21}\}, A_i)$ is nice (assume w.l.o.g. these are $A_1, \dots, A_{\lambda z / 2}$). We now finish the proof of the lemma by showing that these $V^1, \dots V^{21}$ are as desired.

Consider some $V^a, V^b$ with $a,b \in [21]$. Note that $V^a,V^b$ is an $(\varepsilon,p'')$-regular pair in $A^r$. We show a lower bound for $e_{A^r}(V^a, V^b)$. Consider some $A_i$ such that $(\{V^1, \dots, V^{21}\}, A_i)$ is nice. There must be some $W^a, W^b \in \cW_i$ such that for each $x \in \{a,b\}$, $|V^x \cap V(W^x)| \geq \beta m$. By Claim~\ref{claim:intersectionVI}, this implies that at least $\frac{\beta m}{2}$ of the vertices in $V^x$ are in some $I \in W^x$ such that $|I \cap V^x| \geq \frac{C'}{12t}$. Call these $I$'s \emph{good} for $V^x$ and recall that each $I$ is a subset of an independent set $I_j \in \{ I_1, \dots, I_{n_i}\}$ on $C$ vertices in $A_i$. 

For each $I^a$ which is good for $V^a$ and $I^b$ which is good for $V^b$ with $I^a \subset I_j, I^b \subset I_h$, if $j\neq h$, the probability that in $A_i$ there is a complete bipartite graph between $I_j$ and $I_h$ is precisely $p'$. Consider some $I^a, I^b$, which are good for $V^a, V^b$ respectively, with a complete bipartite graph between their respective supersets $I_j$ and $I_h$ in $A_i$. We know that there is no blue $K_{s, s}$ between $I^a \cap V^a$ and $I^b \cap V^b$, so by Theorem~\ref{thm:kovari-sos-turan}, since $\frac{C'}{12t} \gg s$, at least half of the edges between $I^a$ and $I^b$ are red.

Let $X$ be the random variable counting the edges between all pairs $I^a, I^b$ with $I^a \in W^a$ and $I^b \in W^b$ which are good for $V^a$ and $V^b$ respectively and have distinct supersets $I_j \supset I^a$ and $I_h \supset I^b$ in $A_i$, then
$$ \mathbb{E}[X] = \sum_{I^a, I^b} |I^a \cap V^a| |I^b \cap V^b|p' \geq \frac{1}{2} \Big(\frac{\beta m}{2}\Big)^2 p' = \Theta(n\log{n}),$$
where the inequality comes from the fact that at least $\beta m/2$ vertices in each $V^a$ and $V^b$ are in good $I^a$ and $I^b$ respectively; furthermore, for each good $I^a$, there is at most a constant number of vertices in $V^b$ that belong to the same superset $I_j \supset I^a$, and this is accounted for by the factor of $\frac{1}{2}$, which gives a generous lower bound.
Note that at least $X/2$ of these $X$ edges are red. Letting $X_{j,h}$ denote the number of edges between all good pairs $I^a \in W^a, I^b \in W^b$ with $I^a \subset I_j$ and $I^b \subset I_h$, note that $X = \sum_{j \neq h} X_{j,h}$, and we can think of the sum as only going over the pairs $j,h$ such that there exists at least one good pair $I^a, I^b$ with $I^a \subset I_j$ and $I^b \subset I_h$. Since with probability $p'$, $$ X_{j,h} = \sum_{\text{good }I^a \subset I_j, I^b \subset I_h} |I^a \cap V^a| |I^b \cap V^b|,$$
in which case $1 \leq X_{j,h} \leq | I_j \cap V^a | | I_h \cap V^b | \leq C^2$, and $X_{j,h}=0$ otherwise, we can apply Lemma~\ref{lem:weighted-Chernoff}.
We get
$$Pr\Big[X < \mathbb{E}[X]/2 \Big] \leq e^{-\Theta(\mathbb{E}[X])} = e^{-\Theta(n\log{n})}. $$
By a union bound over all at most $2^n$ relevant subsets $R$ of $V(\Gamma)$, all $2^{2n}$ subsets $V^a$ and $V^b$ and all at most $2^C$ subsets $I^a$ of each $I_j \cap V^a$ and subsets $I^b$ of each $I_h \cap V^b$, we have that with probability $1-e^{-\Theta(n\log n)}$ for a fixed $i\in[\lambda z/2]$ we have $X \geq \frac{\beta^2 m^2 p'}{16}$, so the number of red edges between $V^a$ and $V^b$ in $A_i$ is at least $\frac{\beta^2 m^2 p'}{32}$. Thus $d_{A^r_i}(V^a, V^b) \geq \frac{\beta^2 p'}{32}$, where $A^r_i$ denotes the red subgraph of $A_i$.
Moreover, since the probability of failure is sufficiently small, by a union bound we get that $d_{A^r_i}(V^a, V^b) \geq \frac{\beta^2 p'}{32}$ for each $i\in [\lambda z/2]$.

Since each edge is in at most $5$ $A_i$'s by Lemma~\ref{lem:not-too-many-edges-in-more-than-one-ai}, we have
$$ d_{A^r}(V^a,V^b) \geq \frac{\lambda z}{2} \frac{ \beta^2 p'}{5 \cdot 32} \geq \frac{ \lambda \gamma^2 p''}{10^5 \alpha^2}=\Bigg(\frac{q\gamma^2}{800 \alpha^2 e}\Bigg)^{21}\frac{\gamma^2 p''}{10^5 \alpha^2},$$
where we used Claim~\ref{claim:nice-pairs}.
\end{proof}

The following lemma shows that under certain conditions on $A:=A_1 \cup \dots \cup A_z$, regular pairs in $A$ remain regular after subsampling to get $\tilde{A}_1 \cup \dots \cup \tilde{A}_z$. All of those conditions hold with high probability for an outcome of $A_1 \cup \dots \cup A_z$, which we make use of in the choice of our host graph in the proof of Theorem~\ref{thm:main}.
\begin{lemma}
\label{lem:density-q-partition-subsample}
Let $0<\gamma' \ll  \varepsilon \ll \gamma$, and $\mu \gg \gamma'$, and let $K\in \mathbb N$. Let $A$ be an outcome of the random graph distribution $A_1\cup\ldots\cup A_z$ which is $(\gamma',p'')$-upper-uniform. Suppose there is a red/blue colouring of $A$ and $K$ disjoint sets of vertices $V_1, \dots, V_K$ with $|V_i| = \mu n$, such that the red subgraph of $A[V_i, V_j]$ is $(\varepsilon,p'')$-regular with density at least $\gamma p''$, for all $i\neq j$. Furthermore, assume that each $A_i$ has at most $n^2p'$ edges, and that at most $n^{3/2}$ edges are in more than one $A_i$ and no edge is in at least five $A_i$'s.
Then w.h.p the graph $\tilde{A}:=\tilde{A}_1 \cup \dots \cup \tilde{A}_z$ is such that the red subgraph of $\tilde{A}[V_i, V_j]$ is $(4\varepsilon,\tilde{p}'')$-regular with density at least $\frac{\gamma}{2} \tilde{p}''$.

\end{lemma}
\begin{proof}
We refer to the red subgraphs of $A, \tilde{A}, A_i, \tilde{A}_i$ for some $i\in [z]$ as $A^r, \tilde{A}^r, A^r_i, \tilde{A}^r_i$  respectively. Let $U,W$ be distinct sets among $V_1, \dots, V_K$.
We show that the statement is w.h.p.\ satisfied for $U,W$, which together with a union bound over all such pairs completes the proof. 
\begin{claim}
\label{concentration_of_density_ais}
Let $\nu>0$, and let $U' \subseteq U$ and $W' \subseteq W$ with $|U'| \geq \varepsilon |U|$ and $|W'| \geq \varepsilon |W|$. Then w.h.p.\ $d_{\tilde{A}^r}(U', W') = (1 \pm \nu) \frac{\tilde{p}''}{p''}d_{A^r}(U', W')$.
\end{claim}
\begin{proof}
Since $A^r[U, W]$ is an $(\varepsilon, p'')$-regular pair, we have $|d_{A^r}(U', W') - d_{A^r}(U, W)| \leq \varepsilon p''$, so $$e_{A^r}(U', W') \geq p''(\gamma - \varepsilon) |U'| |W'| \geq p''(\gamma - \varepsilon) \varepsilon^2 \mu^2 n^2 = \Omega(p'' n^2) = \Omega(n^{\frac{3}{2} + \delta}\log n).$$

Now, for each $k\in[z]$, let $\cD_k$ denote the set of copies of complete bipartite graphs $K_{C,C}$ formed by two blocks in $A_k$. For each $D\in \cD_k$, denote with $X_D$ the random variable which counts the number of red edges in $\tilde{A}_k$ between $U'$ and $W'$ contained in $D$. Note that by assumption $|\cD_k|\leq e(A_k)\leq n^2p'$. Denote $Y_k=e_{\tilde{A}^r_k}(U',W')$ and observe that $Y_k=\sum_{D\in \cD_k}X_D$.
By Lemma~\ref{edge_probability_a_i}, we have that the expectation of $Y_k$ satisfies
$$ \mathbb{E}[Y_k] = \sum_{e \in E_{A^r_k}(U', W')} Pr\Big[e \in E(\tilde{A}_k)\Big] = e_{A^r_k}(U', W') \frac{\tilde{p}'}{p'}.$$

Furthermore, note that $Y_k$ can be viewed as a random variable on a product of probability spaces, where the coordinates are given by $X_D$, for each $D\in \cD_k$. Observe also that changing one coordinate can change $Y_k$ only by at most $C^2$.

We now want to bound $Y_k$ for each $k$, and for this we have two cases. In the first case, if $e_{A^r_k}(U', W') \leq (\log n)^{-\frac{1}{4}} |U'| |W'| p'$, then clearly we have $Y_k \leq (\log n)^{-\frac{1}{4}} |U'| |W'| p'$. On the other hand, if  $e_{A^r_k}(U', W') \geq (\log n)^{-\frac{1}{4}}|U'| |W'| p'$, by McDiarmid's inequality (Theorem~\ref{thm:mcdiarmid}), and setting $t = \nu \mathbb{E}[Y_k] /10$ we get

$$ Pr\Big[ Y_k \notin (1 \pm \nu/10) \mathbb{E}[Y_k] \Big] \leq \exp\Bigg({-\frac{\Omega(\mathbb{E}[Y_k]^2)}{|\cD_k|C^4}}\Bigg) \leq \exp\Bigg({- \Omega\bigg( \frac{(\log n)^{-\frac{1}{2}}n^4\tilde{p}'^2}{n^2p'}}\bigg)\Bigg) \leq e^{-\Omega(n\sqrt{\log n})}.$$

Now, a union bound over all $2^{2n}$ possible choices of $U', W'$ and $z$ choices of $k$ shows us that w.h.p.\ we have $Y_k \in (1\pm \nu / 10)\mathbb{E}[Y_k]$ for all $U'$, $W'$, and $A_k$ with $e_{A^r_k}(U', W') \geq $ $(\log n)^{-\frac{1}{4}}$ $|U'|$ $ |W'| p'$.
Hence, having in mind the bounds from both cases, we get w.h.p.\ that
\begin{align*}
 e_{\tilde{A}^r}(U',W') &\leq \sum_{k=1}^z e_{\tilde{A}^r_k}(U',W') \leq \sum_{k=1}^z \Bigg(e_{A^r_k}(U', W')\frac{\tilde{p}'}{p'} (1+\nu/10) + (\log n)^{-\frac{1}{4}}n^2 p'\Bigg) \\
 &\leq  (1+\nu/10)\Big(e_{A^r}(U', W') + 4n^{\frac{3}{2}}\Big) \frac{\tilde{p}'}{p'} + o(zn^2p') \leq (1+\nu)e_{A^r}(U', W') \frac{\tilde{p}''}{p''},
\end{align*}
where we used that the sum of all red edges between $U'$ and $W'$ over all the $A_k$'s overcounts $e_{A^r}(U',W')$ by at most $4n^{3/2}$, and the fact that $\frac{\tilde{p}'}{p'} \approx \frac{\tilde{p}''}{p''}$ by Table~\ref{tab:probabilities}. For the lower bound, we let $I\subseteq [z]$ be the set of indices $k$ for which $e_{A^r_k}(U',W') \geq (\log n)^{-\frac{1}{4}} |U'||W'|p'$ to get
\begin{align*}
 e_{\tilde{A}^r} (U', W') \geq& \sum_{k=1}^z e_{\tilde{A}^r_k}(U', W') - 4n^{\frac{3}{2}} \geq \sum_{k\in I} \Bigg((1-\nu/10)e_{A^r_k}(U',W') \frac{\tilde{p}'}{p'} \Bigg)- 4n^{\frac{3}{2}}\\
 \geq&  (1-\nu/10) \frac{\tilde{p}'}{p'}  \Big(e_{A^r}(U',W') - z(\log n)^{-\frac{1}{4}} |U'||W'|p' \Big) - 4n^{\frac{3}{2}}\\
 \geq& (1-\nu/5) \frac{\tilde{p}'}{p'}  e_{A^r}(U',W') \geq (1-\nu) \frac{\tilde{p}''}{p''}  e_{A^r}(U',W') ,
\end{align*}
using the condition on repeated edges across $A_i$'s, and that $\frac{\tilde{p}'}{p'} \approx \frac{\tilde{p}''}{p''}$ again.
\end{proof}
By Claim~\ref{concentration_of_density_ais} we get that $d_{\tilde{A}^r}(U, W) \geq \gamma \tilde{p}'' / 2$. Let $U' \subseteq U$ and $W' \subseteq W$, with $|U'| \geq \varepsilon |U|$ and $|W'| \geq \varepsilon |W|$. We apply Claim~\ref{concentration_of_density_ais} to $U',W'$ and $U,W$ with $\nu:=\varepsilon$, and use the triangle inequality to get
\begin{align*}
&\left|d_{\tilde{A}^r}(U', W') - d_{\tilde{A}^r}(U, W)\right| \\
\leq& \left|d_{\tilde{A}^r}(U', W') - \frac{\tilde{p}''d_{A^r}(U', W')}{p''}\right| + \frac{\tilde{p}''}{p''}\left|d_{A^r}(U', W') - d_{A^r}(U, W)\right| + \left|\frac{\tilde{p}''d_{A^r}(U, W)}{p''} - d_{\tilde{A}^r}(U, W)\right| \\
 \leq& \varepsilon \frac{\tilde{p}''d_{A^r}(U', W')}{p''} + \varepsilon \tilde{p}'' + \varepsilon \frac{\tilde{p}''d_{A^r}(U, W)}{p''} \leq 2 \varepsilon(1+\gamma') \tilde{p}'' + \varepsilon \tilde{p}'' \leq 4\varepsilon \tilde{p}'', 
\end{align*}
where we used that $A$, and therefore also $A^r$, is $(\gamma',p'')$-upper-uniform (for $\gamma' \ll \varepsilon$) in the penultimate inequality, and the fact that $A^r[U, W]$ is $(\varepsilon, p'')$-regular for bounding the middle term on the second line.
\end{proof}

\section{The proof}\label{sec:proof}
After having done a big part of the work in the previous sections, we are ready to put everything together to show Theorem \ref{thm:main}. 

We start by describing the key constants we use. We need the following inequalities to hold: 
\begin{equation}\label{eq:constants}
\{\eta^{-1}, c^{-1}\} \gg C \gg T_3 \gg \varepsilon_3^{-1} \gg C' \gg T_2 \gg  \varepsilon_2^{-1} \gg T_1 \gg \varepsilon^{-1}_1 \gg \ell \gg \delta^{-1}.
\end{equation}
We can think of each $T_i$ as the upper bound on the number of sets we get from an application of the sparse regularity lemma (Lemma~\ref{thm:sparse-regularity-lemma}) with $\varepsilon_i$. Recall that $\delta$ is the constant that determines how close the number of edges $n^{\frac{3}{2} + 2\delta}$ in the host graph $\Gamma$ is to $n^{3/2}$ (see Lemma~\ref{lem:host-graph-edges-upper-bound}).
It is important to choose $\ell$ to be an integer so that $\delta>\frac{1}{4\ell-6}$, since we want $\delta$ to be large enough to be able to embed cycles of length at least $\ell$ later. The constant $C$ is the size of the blocks in $G^C(n,p)$ (see Section~\ref{sec:host}), whereas $cn$ is the number of vertices of the cubic graph $H$, which we are to embed in our $n$-vertex host graph. Finally, as indicated in Section~\ref{subsec:bipartite}, $\eta n$ is the maximum number of vertices not covered by each block matching, and $C'$ is a parameter of the densifiers that we will find (see Definition~\ref{def:densifier}).

Let $H$ be a cubic graph on $cn$ vertices.
We first apply Lemma~\ref{lem:decomposition} to $H$ to obtain a decomposition into induced cycles of length at least $\ell$ and an induced subgraph $J$ with treewidth bounded by $2\ell$. Furthermore, there is a blow-up $\cT$ of a tree $T\boxtimes K_{400\ell}$, where $T$ is of maximum degree $400\ell$, which contains the graph $J$. We also may assume that $v(T)\leq cn$, as we have that $v(H)\leq cn$.

Before diving into the proof, in the following subsection we state some results from \cite{kamcev2021size} and corollaries of them used for embedding monochromatic blow-ups of trees in coloured host graphs. In particular, the host graphs which we use to apply those embedding theorems are the graphs $A'_i \cup M_i$ defined in Section \ref{sec:host}. Furthermore, the results from \cite{kamcev2021size} imply that if we appropriately choose the host graph, then it either contains the required blow-up of a tree in one colour, or it satisfies a certain local density property in the other colour. For completeness, we include the slightly altered proofs from \cite{kamcev2021size} in the appendix.

\subsection{Monochromatic blow-ups of trees in coloured expanders}

To state the necessary results, we will need the following definition.
\begin{defin}\label{def:universal graph}
We say that an $n$-vertex graph $F$ is $\alpha$-joint if for every pair of disjoint sets $S, T \subseteq  V (H)$ with $| S| , | T |  \geqslant  \alpha n$ we have
$e(S, T ) > 0$.
\end{defin}

The following result, Theorem~\ref{thm:kamcev embedding}, can be shown by only slightly modifying the proof in \cite{kamcev2021size}, and for completeness we include its proof in the appendix. It states that every blow-up of a bounded degree tree can be found as a monochromatic copy in a constant blow-up of a third power of an $\alpha$-joint graph. They state the result slightly differently, for a random $D$-regular graph in place of an $\alpha$-joint graph, but this has little effect on the argument. We first need the following definition. 

\begin{defin}
 Let $\mathcal{T}_{n,d}$ be the set of all trees with $n$ vertices and maximum degree at most $d$. Furthermore, let $\mathcal{T}_{n,d}(k)$ be the family of all graphs $T\boxtimes K_k$ where $T\in \mathcal{T}_{n,d}$.
\end{defin}
Note that $\cT$ belongs to $\cT_{cn,400\ell}(400\ell)$.

\begin{restatable}{theorem}{UniversalGraph}{\normalfont(Theorem 3.4 in~\cite{kamcev2021size})\textbf{.}}\label{thm:kamcev embedding}
Let $\alpha \ll c'\ll r^{-1}\ll \{k^{-1},d^{-1}\}$.
Let $\cG$ be an $\alpha$-joint graph on $n$ vertices. 
Then any red/blue colouring of $\cG^3\boxtimes K_{r}$ contains a monochromatic copy of each graph in $\cT_{c'n,d}(k)$. Furthermore, all graphs in $\cT_{c'n,d}(k)$ can be found in $\cG^3\boxtimes K_{r}$ in the same colour.
\end{restatable}

In the rest of the paper, we (evidently) rely on the various parts of our host graph construction, so we refer the reader to Section~\ref{sec:host}, and in particular to Definiton~\ref{defin:the-a-i-s}.

Given a small linear-sized subset $S$ of vertices of our host graph $\Gamma$, the following lemma shows the existence of a copy of $\cT$ in either the red or the blue subgraph of $(A_i'\cup M_i)[S]$.
\begin{lemma}
\label{lem:universal-graph-existence}
Let $L=A_i'\cup M_i$ for some $i$.
The following holds w.h.p.\ for each red/blue colouring of $L$.
Let $S$ be a subset of $V(L)$ of size $|S|=\gamma n$ for $\gamma\gg C^{-1}$. 
Then either the red or the blue subgraph of $L[S]$ contains $\cT$.
\end{lemma}

\begin{proof}
By Lemma~\ref{lem:proportion of A} applied to the cliques in $M_i$ and the set $S$, we get a collection $\mathcal{B}$ of at least $\gamma \frac{n}{2C}$ disjoint cliques $B$ such that $B\subseteq S$ and $|B|= \gamma C/8$, where each $B$ is contained in a distinct clique from $M_i$; here we also used that $M_i$ covers at least $(1-\eta)n$ vertices where $\eta\ll \gamma$ by~(\ref{eq:constants}). Let $F$ be the subgraph of $G_i$ induced by the vertices of $G_i$ corresponding to cliques in $M_i$ which contain a set $B$ from $\mathcal{B}$. 

Since the number of vertices of $F$ is $v(F)=|\mathcal{B}|\geq \frac{\gamma}{2}v(G_i)$, we have that $F$ is $\alpha$-joint for all constants $\alpha$. Indeed, since $G_i$ is a binomial random graph with expected degree logarithmic in its number of vertices, w.h.p.\ every pair of linear-sized subsets of vertices has an edge in between (by a standard Chernoff bound).

Now, look at the copy of the graph $F^3\boxtimes K_{\gamma C/8}$ in $L$ corresponding to the cliques $\cB$. Since $F$ is w.h.p.\ $\alpha$-joint for an arbitrarily small $\alpha$ (in particular also for $\alpha\ll cC/\gamma$), we infer by Theorem~\ref{thm:kamcev embedding}, setting $c'= \frac{cn}{v(F)}\leq \frac{2cC}{\gamma(1-\eta)}$, $r=\gamma C/8$, $k=d=400\ell$, that the considered copy of $F^3\boxtimes K_{\gamma C/8}$ in $L[S]$ either contains a red or a blue copy of $\cT$.
\end{proof}

Now we show a proposition which states that for each $i$, either the blue subgraph of $A_i\cup M_i$ contains $\cT$, or the red subgraph of $A_i\cup M_i$  satisfies a certain local density property. In order to do that, we will need the following definition (which also appears in~\cite{kamcev2021size} and other prior work), together with a theorem which is implicit in \cite{kamcev2021size} and whose proof can be found in our appendix.

\begin{defin}\label{def:aux coloring}
For integers $s$ and $m$, a graph $\cG$ with edge-colouring $\psi:E(\cG)\rightarrow \{\text{red},\text{blue}\}$ and a vertex partition $(V_1, V_2, \dots, V_m)$ of $\cG$, we define the following auxiliary colouring of $K_m$. For vertices $i,j\in [m]$ of $K_m$, the edge $ij$ is coloured blue if the bipartite graph between $V_i$ and $V_j$ in $\cG$ contains a blue $K_{s,s}$, and red otherwise. This edge-colouring is referred to as the \emph{$(\cG,\psi,s)$-colouring of $K_m$.}
\end{defin}

\begin{restatable}{theorem}{qpartition}\normalfont{(Proof of Theorem 3.4 in \cite{kamcev2021size})}\textbf{.}\label{thm:no-univ-vertex-q-partition}
Fix integers $n_0$, $d$, $k$, $q$. Let $s=(d+d^2)k$ and $m \geq 20n_0d^2q$. Let $\cK=T \boxtimes K_k$ for an $n_0$-vertex tree $T$ of maximum degree $d$. Suppose we are given a graph $\cG$, a vertex partition $(V_1, V_2, \dots, V_m)$ of $\cG$, and an edge-colouring $\psi:E(\cG)\rightarrow \{red,blue\}$ such that, for all $i\in[m]$, all the edges of $\cG[V_i]$ are present and are blue, and $|V_i|\geq s$. If $\cG$ does not contain a blue copy of $\cK$, then there is a red copy of a complete $q$-partite graph in the $(\cG,\psi,s)$-colouring of $K_m$, such that every part has size at least $\frac{m}{5d^2q}$.
\end{restatable}

We will also need the following deterministic statement, which does not depend on the outcome of $G_1,\ldots, G_z$. 

\begin{prop}
\label{prop:q-partition-or-cliques}
 Let $\gamma = \frac{1}{10^3T_1}$.  For every $\frac{10^{11} c \ell^2 C'}{\gamma} <\rho < \frac{\gamma}{20}$, the following holds. Let $S$ be a subset of vertices of $\Gamma$ with $|S|= \gamma n$ such that the blue subgraph of $\Gamma[S]$ does not contain $\cT$. Then for each $i$ one of the following is true:
\begin{enumerate}[label=(\roman*)]
    \item \label{item:q-partition} $A_i[S]$ contains a red $(C', \frac{\gamma \rho}{10^{10}\ell^2}, s,q)$-densifier \footnote{See Definition~\ref{def:densifier}.} with $s = (400\ell)^2(400 \ell +1)$ and $q=1000$.
    \item \label{item:red-cliques}
    There are at least $\frac{|S|}{8C}$ cliques $B\in M_i$ such that $|B \cap S|\geq \gamma C/8$ and with some $B' \subseteq B \cap S$ of size at most $ \rho C$ s.t. there is no blue $K_{C'}$ contained in $B\cap S - B'$. 
\end{enumerate}
\end{prop}
\begin{proof}
For each $i\in [z]$, let $L_i = A_i \cup M_i$.
As in the proof of Lemma \ref{lem:universal-graph-existence}, by Lemma~\ref{lem:proportion of A} applied to the cliques in $M_i$ and the set $S$, we get a collection $\mathcal{B}$ of at least $\gamma \frac{n}{2C}$ disjoint cliques $B$ such that $B\in M_i$ and $|B \cap S|\geq \gamma C/8$. Let $F_i$ be the subgraph of $G_i$ induced by the vertices of $G_i$ corresponding to blocks in $\mathcal{B}$. 

Now for each $i\in [z]$, look at the copy of the graph $F_i\boxtimes K_{\gamma C/8}$ in $L_i[S]$, corresponding to the cliques $\cB$ (technically, the cliques in the blow-up are of size at least ${\gamma C/8}$ and not necessarily precisely ${\gamma C/8}$, so we abuse notation slightly here). For each $v \in F_i$, we refer to the copy of $K_{t}$ with $t \geq {\gamma C/8}$ that corresponds to $v$ in $F_i\boxtimes K_{\gamma C/8}$ as $B_i(v)$. Then one of the following occurs for each $i$:
\begin{enumerate}[label=(\alph*)]
    \item \label{item:many-blue}
    At least half of the vertices $v \in F_i$ are such that $B_i(v)$ contains at least $\rho \frac{C}{C'}$ many vertex-disjoint blue copies of $K_{C'}$.
    \item \label{item:not-many-blue} At least half of the vertices $v \in F_i$ are such that $B_i(v)$ contains some $B'_i(v) \subseteq B_i(v)$ of size at most $\rho C$ so that $B_i(v) - B'_i(v)$ has no blue $K_{C'}$.
\end{enumerate}
Indeed, one can remove blue copies of $K_{C'}$ from each $B_i(v)$ repeatedly until either at least $\rho C$ vertices are covered or there are no blue copies of $K_{C'}$ remaining.

If \ref{item:not-many-blue} holds, there are at least $\frac{v(F_i)}{2} \geq \frac{\gamma n}{8C}$ cliques $B_i(v)$, each a subset of a distinct block $B \in M_i$ which fulfills the requirements of \ref{item:red-cliques} with $B'_i(v)$ as $B'$, so we are done. 

Otherwise, if \ref{item:many-blue} holds, we consider the subgraph $F'_i$ of $F_i$ induced by the vertices $v \in F_i$ such that $B_i(v)$ contains at least $\rho \frac{C}{C'}$ many vertex-disjoint blue copies of $K_{C'}$. For each such $v$, let $B'_i(v) \subseteq B_i(v)$ be the subset of size at least $\rho C$ covered by copies of blue $K_{C'}$. Then consider the graph $F'_i \boxtimes K_{\rho C} \subseteq F_i \boxtimes K_{\gamma C/8}$ with the copies of $K_{\rho C}$ corresponding to the $B'_i(v)$'s.
We apply Theorem~\ref{thm:no-univ-vertex-q-partition} to $F'_i \boxtimes K_{\rho C}$ with the blue copies of $K_{C'}$ as a vertex partition $(V_1, \dots, V_{m})$ with $$m:=\frac{v(F_i)}{2}  \frac{\rho C}{C'} \geq \frac{\gamma \rho n}{8 C'}.$$
and with $q=1000$, $d=k=400\ell$, $n_0=cn$, and $s=(d^2+d)k$. We can do this as $m \geq \frac{\gamma \rho n }{8C'} \geq 20 cn (400\ell)^2 \cdot 1000$ and $C' \gg s$, by (\ref{eq:constants}). Since the blue subgraph of $F'_i \boxtimes K_{\rho C}$ does not contain $\cT$ by assumption, there is a red copy of a complete $q$-partite graph in the $(F'_i \boxtimes K_{\rho C},\psi, s)$-colouring of $K_m$, such that every part has size at least $\frac{m}{5d^2q}$, where $\psi$ is the considered colouring of $\Gamma$ restricted to $F'_i \boxtimes K_{\rho C}$. This means that there are $q$ collections $W_1,\dots,W_q$ of blue copies of $K_{C'}$, which are subsets of the copies $B'_i(v)$ of $K_{\rho C}$ and are also pairwise vertex-disjoint (even across different $W_i$'s). Furthermore, the collections $W_1,\dots,W_q$ have the property that $|W_i| = \frac{m}{5d^2q}$ and for every pair $X,Y$ with $X\in W_i$ and $Y\in W_j$ where $i\neq j$, there is no blue $K_{s,s}$ in the complete bipartite graph between $X$ and $Y$. This corresponds precisely to a red $(C', \frac{\gamma \rho}{10^{10}\ell^2}, s,q)$-densifier of $A_i[S]$, since $$|W_i|\geq \frac{m}{5d^2 q} \geq \frac{\gamma \rho n}{40 C' (400\ell)^2 q} \geq \frac{\gamma \rho}{10^{10}\ell^2} \frac{n}{C'}.$$
\end{proof}
\subsection{Embedding the cubic graph}

\begin{proof}[Proof of Theorem~\ref{thm:main}]
    We show that for every $\delta>0$ there exists a $c>0$, such that for every $n$ large enough and every cubic graph $H$ on $cn$ vertices, there is an $n$-vertex host graph with at most $n^{\frac{3}{2} + 2\delta}$ edges that is Ramsey for $H$.

    We first give a summary of the proof. The host graph is as described in Section \ref{sec:host}. We distinguish between two cases, depending on whether each relatively large induced subgraph of the host graph contains a copy of the bounded treewidth part $\cT$ of $H$ in each colour. If that is the case, we have a lot of flexibility and we can afford to embed the rest of $H$ in either colour. We make use of standard techniques to find disjoint vertex sets $V_1, \dots, V_{21}$, all pairs of which are regular and dense in one colour, say red. We use $V_1, \dots, V_{20}$ to embed the long induced cycles, following the strategy from~\cite{conlon2022size}. Finally, we embed $\cT$ in $V_{21}$ in the same colour as the cycles, which is possible by assumption. In the second case there is a relatively large vertex set in the host graph with no copy of $\cT$ in say blue. The challenge then is that we have to embed all of $H$ in red, since $\cT$ may not exist in blue at all (it exists in red due to Lemma~\ref{lem:universal-graph-existence}). Thus, we need to guarantee some density in red. We do that with the help of Proposition~\ref{prop:q-partition-or-cliques}, which allows us to deduce, since there is no blue $\cT$, that there is either a red densifier, or many cliques that do not contain too many blue copies of $K_{C'}$. These are the respective set-ups for Lemmas~\ref{lem:regular-clique-complete-bipartite-graphs} and~\ref{lem:density-from-cliques}, which give us the desired sets $V_1, \dots, V_{21}$ that are pairwise regular and dense in red. From here we can proceed as in the first case.
    
    We start by formally describing our host graph.
    Let $\Gamma$ be an outcome of $G\cup A'_1\cup \ldots \cup A'_z$ as defined in Section \ref{sec:host},
    where the graphs $G_1, \dots, G_z$, which give rise to $A'_1, \dots, A'_z$, satisfy the following conditions: the conclusions of Lemmas~\ref{lem:host-graph-edges-upper-bound},
    \ref{lem:upper-uniformity-ais}, \ref{lem:not-too-many-edges-in-more-than-one-ai}, \ref{lem:regular-clique-complete-bipartite-graphs}, and
    \ref{lem:universal-graph-existence} hold\footnote{As in Footnote~\ref{footnote:constants-of-lemmas}, we specify the constants in the usage of these lemmas later.}, and additionally the following property holds. If we now take $\tilde{A}_1, \dots, \tilde{A}_z$ to be subsampled from $A_1, \dots, A_z$ as in Defition~\ref{defin:subsample-a-is}, then with probability at least $0.9$, the conclusions of both Lemma~\ref{lem:typical_vertices} and Theorem~\ref{thm:embed-in-regular-pairs} hold for $\tilde{A}_1 \cup \dots \cup \tilde{A}_z$. Note that such $G_1, \dots, G_z$ exist by the same argument as the one used for fixing an outcome $G$ of $G^C(n,p)$ in the paragraph before Definition~\ref{defin:the-a-i-s}, except we now use that by Lemma~\ref{lem:a-i-subgraph-of-g-n-p}, $\tilde{A}_1 \cup \dots \cup \tilde{A}_z$ can be coupled as a subgraph of $G(n,\tilde{p}'')$.
    
    Let $U$ be the largest subset of $V(\Gamma)$ such that either the red or the blue subgraph of $\Gamma[U]$ does not contain $\cT$. We distinguish two cases, depending on whether $U$ has size at most $\frac{n}{10^3 T_1}$ or not. For example, if $\Gamma$ is completely red, then we would be in the latter case.
    
	\textbf{Case I}
	
	In the former case, we can embed $H \setminus \cT$ in either colour, since we have sufficiently many and well distributed copies of $\cT$ in both colours. We apply the sparse regularity lemma (Theorem~\ref{thm:sparse-regularity-lemma}) to $\tilde{G}$ with $\varepsilon_1$ and a large enough $t_1$. Note that we can do this since we assumed $G$ is such that with probability at least $0.9$, $\tilde{G}$ is $(\zeta,\tilde{p})$-upper-uniform for all $\zeta>0$ (recall the paragraph above Definition~\ref{defin:the-a-i-s}). Using a standard argument invoking Tur\'an's theorem and Ramsey's theorem (see, for example, the proof of Lemma 19 in \cite{kohayakawa2011sparse}), we obtain 21 sets $V'_1, \dots, V'_{21}$ which are pairwise $(\varepsilon_1,\tilde{p})$-regular and where each pair has density at least $d:=\gamma \tilde{p}$ in say red, where $\gamma = 1/4$. We remove the bad subset for each pair $V_i', V_j'$ via Lemma~\ref{lem:typical_vertices} applied with some $\varepsilon'_1$ such that $\varepsilon_1 \ll \varepsilon'_1 \ll \{\ell^{-1},\gamma\}$ (which is possible by (\ref{eq:constants})). Note that the lemma also guarantees that the density of the pairs does not drop to less than $(1-\varepsilon_1')d$. We then take the subsets given by Lemma~\ref{lem:cleanup} to get $V_1, \dots, V_{21}$, each of size at least $\tilde{n} \geq \frac{n}{2 T_1}$, such that the red subgraph $F$ of $\tilde{G}$ on those vertices satisfies the following:
	
	\begin{itemize}
	\item $d_F(v,V_j)\geq \tilde{n}d(1-\varepsilon_1')/2$ for each $v \in V_i$ and $i \neq j $.
	\item for all distinct $i,j,h,g$ (but possibly $h=g$), for each $v \in V_h, w \in V_g$, and any $N_1 \subseteq N_{F}(v, V_i)$ and $N_2 \subseteq N_{F}(w,V_j)$
     of size $|N_1|=|N_2|= \frac{\tilde{n}d}{20}$,
     $(N_1, N_2)$ and $(N_1, V_j)$ are $(\varepsilon',p)$-regular of density at least $\frac{d}{2}$ in $F$.
	\end{itemize}
	
	The first step of our embedding procedure is to embed the graph $\cT$ in the red subgraph of $\Gamma[V_{21}]$, which we can do since $V_{21}$ is of size at least $\frac{n}{2T_1}$, so $\Gamma[V_{21}]$ contains both a red and a blue copy of $\cT$ by the assumption of Case I.
    What is left is to use the remaining vertices in $V_1, \dots, V_{20}$ to embed the induced cycles from the decomposition. 
    By definition, those induced cycles are such that every vertex in each of them has at most one neighbour in the previously embedded part of $H$. So for each such vertex, the 'candidate set' (i.e. the set where this vertex can be embedded) in each $V_i$ is of size at least $\tilde{n}d(1-\varepsilon_1')/2$. Now we use the technique developed in \cite{conlon2022size} to embed those graphs in $V_1, \dots, V_{20}$. The only difference in our case is that we start the embedding process with some candidate sets of vertices which have a neighbour in a graph from $\cT$, which is precisely the set-up for using Theorem~\ref{thm:embed-in-regular-pairs}. Recall that by our choice of $G$ in the paragraph above Definition~\ref{defin:the-a-i-s}, Theorem~\ref{thm:embed-in-regular-pairs} is applicable to $\tilde{G}$ with probability at least $0.9$. We apply it with $\varepsilon:=\varepsilon'_1$, noting that $\{\ell^{-1},\gamma\} \gg \varepsilon_1' \gg \varepsilon_1$, which finishes the proof in this case.

	\textbf{Case II}
	
	In the latter case, there is some $U$ of size $|U| > \frac{n}{10^3 T_1}$ such that either the red subgraph or the blue subgraph of $\Gamma[U]$ does not contain $\cT$. Suppose w.l.o.g. that this holds for the blue subgraph of $\Gamma[U]$. Note that by Lemma~\ref{lem:universal-graph-existence}, any subset $U'$ of $U$ of size $\phi n$ with $\phi \gg C^{-1}$ is such that the red subgraph of $\Gamma[U']$ contains $\cT$.
	
	We now restrict ourselves to a subset $V'$ of $U$ that has size precisely $\phi n$ where $\phi=\frac{1}{10^3 T_1}$.
	Since the blue subgraph of $\Gamma[V']$ does not contain $\cT$, we can apply Proposition~\ref{prop:q-partition-or-cliques}. We do so with $S:=V'$ and $\rho$ such that $\varepsilon_2 \ll \rho \ll T_1^{-1}$. Note that the conditions for $\rho$ can be satisfied by~(\ref{eq:constants}). We get that for each $i $, one of the following holds
	\begin{enumerate}[label=(\alph*)]
	\item \label{a} $A_i[V']$ contains a red $(C', \frac{\rho}{10^{13}\ell^2 T_1}, s,q)$-densifier with $s =(400\ell)^2(400\ell+1)$ and $q=1000$
	\item \label{b} There are at least $|V'|/8C$ cliques $B \in M_i$ such that $|B \cap V'| \geq \frac{C}{ 8 \cdot 10^3 T_1}$ and with some $B' \subseteq B \cap V'$ of size at most $\rho C$ s.t. there is no blue $K_{C'}$ in $B\cap V' - B'$. 
	\end{enumerate}
	We again distinguish between two cases, depending on whether at least $z/2$ of indices $i\in[z]$ satisfy~\ref{a} or~\ref{b} (recall that $z$ is the number of $A_i$'s).
	
	If \ref{a} is more common, we apply Lemma~\ref{lem:regular-clique-complete-bipartite-graphs} with $R:=V'$, $\alpha:= \frac{1}{ 10^3 T_1}$, $\gamma:=\frac{\rho}{10^{13} \ell^2 T_1}$, $q:=1000$, $\mu:=T_2^{-1}$, $s:=(400\ell)^2(400\ell+1)$ and $\varepsilon := \varepsilon_2$, which we can do since $C' \gg T_2$ and $T_2^{-1} \ll \varepsilon_2 \ll T_1^{-1}$. We get $V'_1, \dots, V'_{21}$ of size at least $\mu n$, such that all pairs are $(\varepsilon_2, p'')$-regular in the red subgraph of $A_1 \cup \dots \cup A_z$ with density at least $\tau p''$ where $\tau = \tau(T_1,\ell) \gg \varepsilon_2$. 
 Since the outcome of Lemma~\ref{lem:upper-uniformity-ais} holds, the graph $A_1 \cup \dots \cup A_z$ is $(\gamma',p'')$-upper-uniform for every constant $\gamma'>0$, and furthermore each $A_i$ has at most $n^2p'$ edges by standard concentration bounds (as the expected number of its edges is $\binom{n_i}{2}p' C^2<n^2p'/2$).
  Hence we can apply Lemma~\ref{lem:density-q-partition-subsample} to $V'_1, \dots, V'_{21}$ and get that all pairs are $(4\varepsilon_2, \tilde{p}'')$-regular with density at least $\tau \tilde{p}''/2$ in the red subgraph of $\tilde{A}_1 \cup \dots \cup \tilde{A}_z$.
	
	If \ref{b} is more common, we apply Lemma~\ref{lem:density-from-cliques} (possibly taking supersets of the sets $B'$) with $R:=V'$, $\alpha:=\frac{1}{10^3 T_1}$, $\mu:=T_3^{-1}$, $\varepsilon:=\varepsilon_3$,
	which we can do since $\{T_1, C'\} \ll \varepsilon_3^{-1} \ll \{C,T_3\}$. We get $V'_1, \dots, V'_{21}$ of size $\mu n$, such that all pairs are $(\varepsilon_3, p)$-regular in the red subgraph of $G$ with density at least $\tau p$, where $\tau = \tau(C', T_1)\gg \varepsilon_3$. Similarly to the previous paragraph, since $G$ is $(\gamma', p)$-upper-uniform for fixed $\gamma'>0$ by Lemma~\ref{lem:upper-uniformity}, we can now apply Lemma~\ref{lem:g-c-n-p-subsample-regular-sets} to get that all pairs are $(4\varepsilon_3,\tilde{p})$-regular with density at least $\tau \tilde{p} / 2$ in the red subgraph of $\tilde{G}$.
	
	In both cases, we can now proceed as in Case I, substituting $\tilde{p}''$ for $\tilde{p}$ if \ref{a} is more common. Furthermore, the density of the regular pairs in red is now $\tau \tilde{p}'' / 2$ if \ref{a} is more common or $\tau \tilde{p} / 2$ if \ref{b} is more common. 
	This density is much larger than $\varepsilon_2\tilde{p}''$ and $\varepsilon_3\tilde{p}$ respectively, which enables us to use the same embedding strategy as in Case I. For completeness, we provide the details below. Let $\varepsilon:=\varepsilon_2$ and $\pi:=\tilde{p}''$ if \ref{a} is more common and $\varepsilon:=\varepsilon_3$ and $\pi:=\tilde{p}$ if \ref{b} is more common, and consider the appropriate $\tau$ and $\mu$, where $d:=\tau \pi$ is the density of the regular pairs and $\mu n$ is a lower bound on the size of the sets $V_i$. Finally, let $G_0 = (\tilde{A}_1 \cup \dots \cup \tilde{A}_z)[V']$ if \ref{a} is more common and $G_0 = \tilde{G}[V']$ if \ref{b} is more common.
	
	We remove the bad subset for each pair $V'_i, V'_j$ via Lemma~\ref{lem:typical_vertices} applied with $\varepsilon'$ such that $\varepsilon \ll \varepsilon' \ll \{\ell^{-1},\tau\}$, and take the subsets given by Lemma~\ref{lem:cleanup} to get $V_1, \dots, V_{21}$, each of size at least $\tilde{n} := \frac{\mu n}{2 }$, such that the red subgraph $F$ of $G_0$ on those vertices satisfies the following:
	
	\begin{itemize}
	\item $d_F(v,V_j)\geq \tilde{n}d(1-\varepsilon')/2$ for each $v \in V_i$ with $i\neq j$.
	\item for all distinct $i,j,h,g$ (but possibly $h=g$), for each $v \in V_h, w \in V_g$, and any $N_1 \subseteq N_{F}(v, V_i)$ and $N_2 \subseteq N_{F}(w,V_j)$
     of size $|N_1|=|N_2|= \frac{\tilde{n}d}{20}$,
     $(N_1, N_2)$ and $(N_1, V_j)$ are $(\varepsilon',\pi)$-regular of density at least $\frac{d}{2}$ in $F$.
	\end{itemize}
	
	Recall that any subset $U'\subseteq V'$ of size $\phi' n$ with $\phi' \gg C^{-1}$ is such that the red subgraph of $\Gamma[U']$ contains $\cT$. Since $|V_{21}| \geq \frac{\mu n}{2} \gg C^{-1} n$, we can embed $\cT$ in the red subgraph of $\Gamma[V_{21}]$. What is left is to use the remaining vertices in $V_1, \dots, V_{20}$ to embed the induced cycles from the decomposition. For each vertex, the `candidate set' (i.e. the set where this vertex can be embedded) in each $V_i$ is of size at least $\tilde{n}d(1-\varepsilon')/2$. Due to our assumptions on $G$ and $G_1, \dots, G_z$ from the paragraph above Definition~\ref{defin:the-a-i-s} and the second paragraph of the current proof, with probability at least $0.9$, the graph $G_0$ and its red subgraph $F$ are such that we can apply Theorem \ref{thm:embed-in-regular-pairs} to them. We do so with $\varepsilon:=\varepsilon'$, noting that $\{\ell^{-1}, \tau\} \gg \varepsilon' \gg \varepsilon$, which finishes the proof.
\end{proof}

\section{Concluding remarks}\label{sec:concluding}

In this paper we have shown that the size-Ramsey number of $n$-vertex cubic graphs is of order $O\left(n^{3/2+o(1)}\right)$. In fact, our proof gives a stronger universality result---for any 2-colouring of the $n^{3/2+o(1)}$ edges of our host graph, there is a colour class which contains all cubic graphs on $n$ vertices. On the other hand, it is known that any graph which contains all $n$-vertex cubic graphs must have $\Omega(n^{4/3})$ edges, even without colouring (see~\cite{alon2000universality}). Hence the optimal partition universal graph for the class of all $n$-vertex cubic graphs has at least $\Omega(n^{4/3})$ and at most $n^{3/2+o(1)}$ edges, and it is not clear to us where the truth lies. 

Going back to size-Ramsey numbers of $n$-vertex cubic graphs $H$, it might be true that in general $\hat{r}(H)=o(n^{4/3})$, but in that case an upper bound proof would require several distinct host graph constructions for different cubic graphs $H$. But it is even not completely clear that a general upper bound of $o(n^{3/2})$ should hold. Our proof technique reaches certain hard natural barriers, the most significant one being that at density $p=o(n^{-1/2})$, at least in `uniformly' dense graphs, regularity inheritance between the candidate sets is no longer guaranteed. That was essential to our approach, as we relied on the regularity method and the K{\L}R conjecture to embed cycles into the host graph. Therefore, if possible, pushing the upper bound below $n^{3/2}$ would certainly require new ideas and a different approach. Let us note here that in the special case of the grid graph, Conlon, Nenadov and Truji\'c \cite{conlon2022grids} managed to overcome the regularity inheritance barrier, by using a host graph tailored for the grid graph, exploiting its structural properties. In particular, parts of their host graph are locally very dense, so they are able to use regularity inheritance locally, and get away with using much smaller global density. They get the bound of $\hat{r}(H)=O(n^{5/4})$ where $H$ is the grid graph on $n$ vertices. Finally, for the class of cubic graphs, one could alternatively hope to avoid using regularity inheritance, but this would most probably require entirely new embedding techniques. 

Recall that Kohayakawa, R\"odl, Schacht and Szemer\'edi~\cite{kohayakawa2011sparse} showed that $\hat{r}(H) \leq n^{2 - 1/\Delta + o(1)}$ for all $n$-vertex graphs $H$ with maximum degree $\Delta$, which was improved to $\hat{r}(H) \leq n^{2 - \frac{1}{\Delta-1/2} + o(1)}$ by Nenadov~\cite{nenadov2016ramsey} in the special case when $H$ does not contain a triangle and when $\Delta\geq 5$. If one tries to generalize our approach to arbitrary bounded $\Delta$ to show a bound of $\hat{r}(H) \leq n^{2 - 1/(\Delta-1) + o(1)}$ with the appropriate modifications, everything goes through, except for the regularity inheritance of the candidate sets. More precisely, the candidate sets are now the common neighbourhoods of tuples of already embedded vertices, and hence it is significantly harder to make sure that those common neighbourhoods behave well in the sense of regularity inheritance, even though they will typically be of large enough size if one assumes edge probability $p=n^{-1/(\Delta-1)+o(1)}$. It is quite possible that by embedding the parts from the decomposition more carefully, one can control the choice of tuples so that regularity is still inherited, but we chose not to pursue this in this paper. It would certainly be interesting to see if this can be done.

Our proof does not extend to more than two colours, primarily due to the reliance on the machinery in \cite{kamcev2021size} for bounded treewidth embeddings, which is similarly restricted to two colours. Utilizing tools from \cite{berger2019size}, along with further technical adjustments, could be a plausible direction for generalizing the result to more colours.

\paragraph*{Acknowledgements.} We would like to thank Rajko Nenadov and Milo\v{s} Truji\'{c} for helpful comments and discussions. We are grateful to the anonymous reviewers for their insightful feedback, which has significantly improved the exposition of this paper.

{\small \bibliographystyle{abbrv} \bibliography{thebib}}

\begin{thebibliography}{10}

\bibitem{alon2000universality}
N.~Alon, M.~Capalbo, Y.~Kohayakawa, V.~R\"odl, A.~Ruci{\'n}ski, and
  E.~Szemer{\'e}di.
\newblock Universality and tolerance.
\newblock In {\em Proceedings 41st Annual Symposium on Foundations of Computer
  Science}, pages 14--21. IEEE, 2000.

\bibitem{alon2016probabilistic}
N.~Alon and J.~H. Spencer.
\newblock {\em {The Probabilistic Method}}.
\newblock Hoboken, NJ: John Wiley \& Sons, 4th edition, 2016.

\bibitem{balogh2015independent}
J.~Balogh, R.~Morris, and W.~Samotij.
\newblock Independent sets in hypergraphs.
\newblock {\em J. Am. Math. Soc.}, 28(3):669--709, 2015.

\bibitem{beck1983on}
J.~Beck.
\newblock On size {Ramsey} number of paths, trees, and circuits. {I}.
\newblock {\em J. Graph Theory}, 7:115--129, 1983.

\bibitem{berger2019size}
S.~Berger, Y.~Kohayakawa, G.~S. Maesaka, T.~Martins, W.~Mendon{\c{c}}a, G.~O.
  Mota, and O.~Parczyk.
\newblock The size-{Ramsey} number of powers of bounded degree trees.
\newblock {\em J. Lond. Math. Soc.}, 103(4):1314--1332, 2021.

\bibitem{bodlaender1997treewidth}
H.~L. Bodlaender and D.~M. Thilikos.
\newblock Treewidth for graphs with small chordality.
\newblock {\em Discrete Appl. Math.}, 79(1-3):45--61, 1997.

\bibitem{chvatal1983ramsey}
V.~Chvat{\'a}l, V.~R{\"o}dl, E.~Szemer{\'e}di, and W.~T.~j. Trotter.
\newblock The {Ramsey} number of a graph with bounded maximum degree.
\newblock {\em J. Comb. Theory, Ser. B}, 34:239--243, 1983.

\bibitem{clemens2019size}
D.~Clemens, M.~Jenssen, Y.~Kohayakawa, N.~Morrison, G.~O. Mota, D.~Reding, and
  B.~Roberts.
\newblock The size-{Ramsey} number of powers of paths.
\newblock {\em J. Graph Theory}, 91(3):290--299, 2019.

\bibitem{clemens2021size}
D.~Clemens, M.~Miralaei, D.~Reding, M.~Schacht, and A.~Taraz.
\newblock {On the size-Ramsey number of grid graphs}.
\newblock {\em Combinatorics, Probability and Computing}, 30(5):670--685, 2021.

\bibitem{conlon2009a}
D.~Conlon.
\newblock A new upper bound for diagonal {Ramsey} numbers.
\newblock {\em Ann. Math. (2)}, 170(2):941--960, 2009.

\bibitem{conlon2021three}
D.~Conlon, J.~Fox, and Y.~Wigderson.
\newblock {Three early problems on size Ramsey numbers}.
\newblock {\em arXiv preprint arXiv:2111.05420}, 2021.

\bibitem{conlon2014klr}
D.~Conlon, W.~T. Gowers, W.~Samotij, and M.~Schacht.
\newblock {On the K{\L}R conjecture in random graphs}.
\newblock {\em Israel Journal of Mathematics}, 203(1):535--580, 2014.

\bibitem{conlon2022grids}
D.~Conlon, R.~Nenadov, and M.~Truji{\'c}.
\newblock {On the size-Ramsey number of grids}.
\newblock {\em arXiv preprint arXiv:2202.01654}, 2022.

\bibitem{conlon2022size}
D.~Conlon, R.~Nenadov, and M.~Truji{\'c}.
\newblock The size-{R}amsey number of cubic graphs.
\newblock {\em Bulletin of the London Mathematical Society}, 54(6):2135--2150,
  2022.

\bibitem{ding1995some}
G.~Ding and B.~Oporowski.
\newblock Some results on tree decomposition of graphs.
\newblock {\em Journal of Graph Theory}, 20(4):481--499, 1995.

\bibitem{draganic2021size}
N.~Dragani{\'c}, M.~Krivelevich, and R.~Nenadov.
\newblock {The size-Ramsey number of short subdivisions}.
\newblock {\em Random Structures \& Algorithms}, 59(1):68--78, 2021.

\bibitem{draganic2022rolling}
N.~Dragani{\'c}, M.~Krivelevich, and R.~Nenadov.
\newblock Rolling backwards can move you forward: on embedding problems in
  sparse expanders.
\newblock {\em Trans. Am. Math. Soc.}, 375(7):5195--5216, 2022.

\bibitem{erdos1947some}
P.~Erd{\H{o}}s.
\newblock Some remarks on the theory of graphs.
\newblock {\em Bull. Am. Math. Soc.}, 53:292--294, 1947.

\bibitem{erdos1978size}
P.~Erd{\H{o}}s, R.~J. Faudree, C.~C. Rousseau, and R.~H. Schelp.
\newblock The size ramsey number.
\newblock {\em Periodica Mathematica Hungarica}, 9(1-2):145--161, 1978.

\bibitem{erdos1935a}
P.~Erd{\H{o}}s and G.~Szekeres.
\newblock A combinatorial problem in geometry.
\newblock {\em Compos. Math.}, 2:463--470, 1935.

\bibitem{friedman1987expanding}
J.~Friedman and N.~Pippenger.
\newblock Expanding graphs contain all small trees.
\newblock {\em Combinatorica}, 7:71--76, 1987.

\bibitem{gerencser1967ramsey}
L.~Gerencs\'er and A.~Gy\'arf\'as.
\newblock On {Ramsey}-type problems.
\newblock {\em Ann. Univ. Sci. Budap. Rolando E{\"o}tv{\"o}s, Sect. Math.},
  10:167--170, 1967.

\bibitem{han2020multicolour}
J.~Han, M.~Jenssen, Y.~Kohayakawa, G.~O. Mota, and B.~Roberts.
\newblock The multicolour size-{Ramsey} number of powers of paths.
\newblock {\em J. Comb. Theory, Ser. B}, 145:359--375, 2020.

\bibitem{han2021size}
J.~Han, Y.~Kohayakawa, S.~Letzter, G.~Oliveira~Mota, and O.~Parczyk.
\newblock The size-{Ramsey} number of 3-uniform tight paths.
\newblock {\em Adv. Comb.}, 2021:12, 2021.
\newblock Id/No 5.

\bibitem{haxell1995the}
P.~E. Haxell, Y.~Kohayakawa, and T.~{\L}uczak.
\newblock The induced size-{Ramsey} number of cycles.
\newblock {\em Comb. Probab. Comput.}, 4(3):217--239, 1995.

\bibitem{kamcev2021size}
N.~Kam\v{c}ev, A.~Liebenau, D.~R. Wood, and L.~Yepremyan.
\newblock The size {Ramsey} number of graphs with bounded treewidth.
\newblock {\em SIAM J. Discrete Math.}, 35(1):281--293, 2021.

\bibitem{kohayakawa1997szemeredi}
Y.~Kohayakawa.
\newblock Szemer{\'e}di's regularity lemma for sparse graphs.
\newblock In {\em Foundations of computational mathematics. Selected papers of
  a conference, held at IMPA in Rio de Janeiro, Brazil, January 1997}, pages
  216--230. Berlin: Springer, 1997.

\bibitem{kohayakawa2003szemeredi}
Y.~Kohayakawa and V.~R{\"o}dl.
\newblock Szemer{\'e}di's regularity lemma and quasi-randomness.
\newblock In {\em Recent advances in algorithms and combinatorics}, pages
  289--351. New York, NY: Springer, 2003.

\bibitem{kohayakawa2011sparse}
Y.~Kohayakawa, V.~R{\"o}dl, M.~Schacht, and E.~Szemer{\'e}di.
\newblock Sparse partition universal graphs for graphs of bounded degree.
\newblock {\em Advances in Mathematics}, 226(6):5041--5065, 2011.

\bibitem{Kosowski2015chordal}
A.~Kosowski, B.~Li, N.~Nisse, and K.~Suchan.
\newblock {$k$}-chordal graphs: from cops and robber to compact routing via
  treewidth.
\newblock {\em Algorithmica}, 72(3):758--777, 2015.

\bibitem{kovari1954problem}
T.~K{\"o}v{\'a}ri, V.~T. S{\'o}s, and P.~Tur{\'a}n.
\newblock On a problem of {K}. {Zarankiewicz}.
\newblock {\em Colloq. Math.}, 3:50--57, 1954.

\bibitem{letzter2021size}
S.~Letzter, A.~Pokrovskiy, and L.~Yepremyan.
\newblock {Size-Ramsey numbers of powers of hypergraph trees and long
  subdivisions}.
\newblock {\em arXiv preprint arXiv:2103.01942}, 2021.

\bibitem{mcdiarmid1989method}
C.~McDiarmid.
\newblock On the method of bounded differences.
\newblock {\em Surveys in combinatorics}, 141(1):148--188, 1989.

\bibitem{nenadov2016ramsey}
R.~Nenadov.
\newblock {\em Ramsey and universality properties of random graphs}.
\newblock PhD thesis, ETH Z\"urich, 2016.

\bibitem{nenadov2022new}
R.~Nenadov.
\newblock {A new proof of the K{\L}R conjecture}.
\newblock {\em Advances in Mathematics}, 406:108518, 2022.

\bibitem{pippenger1989asymptotic}
N.~Pippenger and J.~Spencer.
\newblock Asymptotic behavior of the chromatic index for hypergraphs.
\newblock {\em Journal of Combinatorial Theory, Series A}, 51(1):24--42, 1989.

\bibitem{ramsey1929on}
F.~P. Ramsey.
\newblock On a problem of formal logic.
\newblock {\em Proc. Lond. Math. Soc. (2)}, 30:264--286, 1929.

\bibitem{rodl1993lower}
V.~R{\"o}dl and A.~Ruci{\'n}ski.
\newblock Lower bounds on probability thresholds for {Ramsey} properties.
\newblock In {\em Combinatorics, Paul Erd\H{o}s is eighty. Vol. 1}, pages
  317--346. Budapest: J{\'a}nos Bolyai Mathematical Society, 1993.

\bibitem{rodl1995threshold}
V.~R{\"o}dl and A.~Ruci{\'n}ski.
\newblock Threshold functions for {Ramsey} properties.
\newblock {\em J. Am. Math. Soc.}, 8(4):917--942, 1995.

\bibitem{rodl2000size}
V.~R{\"o}dl and E.~Szemer{\'e}di.
\newblock On size {Ramsey} numbers of graphs with bounded degree.
\newblock {\em Combinatorica}, 20(2):257--262, 2000.

\bibitem{sah2020diagonal}
A.~Sah.
\newblock {Diagonal Ramsey via effective quasirandomness}.
\newblock {\em arXiv preprint arXiv:2005.09251}, 2020.

\bibitem{saxton2015hypergraph}
D.~Saxton and A.~Thomason.
\newblock Hypergraph containers.
\newblock {\em Invent. Math.}, 201(3):925--992, 2015.

\bibitem{spencer1975ramsey}
J.~Spencer.
\newblock Ramsey's theorem - a new lower bound.
\newblock {\em J. Comb. Theory, Ser. A}, 18:108--115, 1975.

\bibitem{spencer1978asymptotic}
J.~Spencer.
\newblock Asymptotic lower bounds for {Ramsey} functions.
\newblock {\em Discrete Math.}, 20:69--76, 1978.

\bibitem{tikhomirov2024bounded}
K.~Tikhomirov.
\newblock On bounded degree graphs with large size-ramsey numbers.
\newblock {\em Combinatorica}, 44(1):9--14, 2024.

\bibitem{turan1941external}
P.~Tur{\'a}n.
\newblock On an external problem in graph theory.
\newblock {\em Mat. Fiz. Lapok}, 48:436--452, 1941.

\bibitem{wilson1975existence}
R.~M. Wilson.
\newblock An existence theory for pairwise balanced designs, {III}: Proof of
  the existence conjectures.
\newblock {\em Journal of Combinatorial Theory, Series A}, 18(1):71--79, 1975.

\bibitem{wood2009tree}
D.~R. Wood.
\newblock On tree-partition-width.
\newblock {\em European Journal of Combinatorics}, 30(5):1245--1253, 2009.

\end{thebibliography}

\appendix
\section{Appendix}
\label{sec:appendix}
For completeness, we provide details on some tools we use in the proof of our main theorem. Most of the exposition here follows closely~\cite{kamcev2021size}, but we show the adjustments necessary for our applications.

We make use of a well-known result by Friedman and Pippenger~\cite{friedman1987expanding}. For a graph $H$ and $X\subseteq V(H)$, let $\Gamma_H(X)$ be the set of vertices in $V(H)$  
adjacent to some vertex in $X$. We say that a graph $H$ is $(s,d)$-\emph{expanding} if for every set $X\subseteq V(H)$ with $1\leq |X|\leq s$, it holds that $|\Gamma_H(X)|\geq d|X|$.

\begin{lemma}[\cite{friedman1987expanding}]\label{lem:friedman} 
If $H$ is a non-empty $(2n-2,d+1)$-expanding graph, then it contains every tree with $n$ vertices and maximum degree at most $d$.
\end{lemma}

The following lemma shows that if all sets of certain size $s$ expand well, then one can remove a small number of vertices to obtain a graph where all sets of size at most $s$ expand well.
\begin{lemma} [Lemma 3.1. in \cite{draganic2022rolling}]\label{lem:big sets expand}
Let $G$ be a graph such that $|\Gamma_G(X)| \geq  3Ks$ for every subset $X \subseteq V (G)$ of size
$|X| = s$, for some $s \in \mathbb{N}$ and $K \geq 1$. Then there exists a subset $B \subseteq V (G)$ of size $|B| < s$
such that $G - B$ is $(s, K)$-expanding.
\end{lemma}

The next lemma shows a connection between $\alpha$-joint graphs (recall Definition~\ref{def:universal graph}) and expanding graphs.
\begin{lemma}\label{lem:joint expands} 
Let $G$ be an $n$-vertex $\alpha$-joint graph for some $\alpha>0$ and let $d\geq1$. Then every induced subgraph of $G$ on at least $10\alpha dn$ vertices contains a non-empty $(\alpha n, d)$-expanding subgraph.
\end{lemma}
\begin{proof}
Recall that since $G$ is $\alpha$-joint (Definition~\ref{def:universal graph}), every pair of disjoint vertex sets of size at least $\alpha n$ have an edge between them.
Let $G'$ be an induced subgraph of $G$ on at least $10\alpha dn$ vertices. Every subset $S$ of $\alpha n$ vertices of $G'$ has at least  $10\alpha dn-2\alpha n$ neighbours in $G'$, as there can be only be at most $\alpha n$ vertices outside of $S$ in $G'$ without a neighbour in $S$, since $G$ is $\alpha$-joint. 
So all sets $S$ of size $\alpha n$ have $|\Gamma_{G'}(S)|\geq \frac{10\alpha dn-2\alpha n}{\alpha n}|S|\geq\ 8d|S|$.
Now, by Lemma~\ref{lem:big sets expand}, there is a subgraph of $G'$ on at least $10\alpha dn-\alpha n$ vertices, which is $(\alpha n,d)$-expanding. 
\end{proof}


We also need a lemma from~\cite{kamcev2021size}, which gives a dichotomy in $K_N$ between the containment of all trees in $\cT_{n,d}$ on the one hand, and a complete $q$-partite graph on the other hand, for appropriately chosen $N$ with respect to $n,d,q$.
\begin{lemma}[Lemma 3.1 in~\cite{kamcev2021size}]\label{aux:treeorpartite}
Fix integers $n, d, q$ and let  $N \geq 20ndq$. In every red/blue-colouring of $E(K_N)$ there is either a blue copy of every tree in $\mathcal{T}_{n,d}$, or a red copy of a complete $q$-partite graph in which every part has size at least $\frac{N}{5dq}$. 
\end{lemma}

The following definition is precisely the same as in~\cite{kamcev2021size}.
For a tree $T$ with root $r$, define the \emph{truncation} $T'$ of $T$ as the tree obtained from $T$ by removing each vertex $v$ at a positive even distance from $r$, and for each such $v$, adding an edge from the parent of $v$ to each child of $v$ in $T$. Observe that the maximum degree of $T'$ is at most $d^2$, where $d:= \Delta(T)$.

The next lemma shows that, if $G[V_i]$ is a blue clique for all $i\in[m]$, the existence of a blue tree in the $(G,\psi,s)$-colouring of $K_m$ implies there is a blue blow-up of a related tree in $G$.
\begin{lemma}[Lemma 3.2 in~\cite{kamcev2021size}]~\label{lem:chopping}
Fix integers $n_0$, $d$, $k$, $m$. Let $T$ be a tree in $\cT_{n_0,d}$ rooted at $x_0$, and let $T'$ be the truncation of $T$. Let $s=(d+d^2)k$. Suppose we are given a graph $G$, a vertex partition $(V_1, V_2, \dots, V_m)$ of $G$, and an edge-colouring $\psi:E(G)\rightarrow \{\text{red},\text{blue}\}$ such that, for all $i\in[m]$, all the edges of $G[V_i]$ are present and are blue, and $|V_i|\geq s$. If there exists a blue copy of $T'$
in the $(G,\psi,s)$-colouring of $K_m$, then there exists a blue copy of $T \boxtimes K_k$ in $G$.
\end{lemma}
The next theorem, which we use in the proof of Proposition~\ref{prop:q-partition-or-cliques}, combines the two previous lemmas to show that if a colouring of a certain blow-up does not contain a blow-up of some bounded degree tree in one colour, then it is dense in the other colour. Recall the definition of a $(\cG,\psi,s)$-colouring (Definition~\ref{def:aux coloring}).
\qpartition*
\begin{proof}
Fix an arbitrary root $x_0$ of $T$, and let $T'$ be the truncation of $T$. By Lemma~\ref{lem:chopping}, since there is no blue copy of $T \boxtimes K_k$ in $\cG$, there is no blue copy of $T'$ in the $(\cG, \psi, s)$-colouring of $K_m$. Note that $T'$ belongs to $\mathcal{T}_{n_0, d^2}$. Now Lemma~\ref{aux:treeorpartite} applied to $K_m$ tells us that there is a red copy of a complete $q$-partite graph in which every part has size at least $\frac{m}{5d^2q}$.
\end{proof}

For a graph $F$, we denote by $F\{t\}$ the graph obtained from $F$ by replacing each vertex $v$ by an independent set $I(v)$ of size $t$, and every edge $vw$ by a complete bipartite graph between the sets $I(v)$ and $I(w)$.

\begin{lemma}[Lemma 3.3 in \cite{kamcev2021size}]~\label{lem:lifting}
Fix $t\geq 1$. Let $F$ be a graph with maximum degree $\Delta$. Let $F'$ be a spanning subgraph of $F\{t\}$  such that for every edge $vw\in E(F)$ there are at least $(1-\frac{1}{8\Delta})t^2$ edges in $F'$ between $I(v)$ and $I(w)$. Then $F\subseteq F'$.
\end{lemma}

With all these ingredients at hand, we are now ready to show a modified version of a theorem in~\cite{kamcev2021size}, which we use in the proof of Lemma~\ref{lem:universal-graph-existence}.
\UniversalGraph*
\begin{proof}
Let $t$ be a constant such that $\{k,d\}\ll t\ll r$, and let $s = (d^2 + d)k$.
 Let $A(v)$ be the copy of $K_{r}$ that corresponds to $v\in V(\cG)$. Denote $\cG^3 \boxtimes K_r$ by $G$. Fix an edge-colouring $\psi:E(G)\rightarrow \{\text{red}, \text{blue}\}$ of $G$.

Since we can assume that $r$ is at least the Ramsey number $r(t)$, for every $v\in V(\cG)$ we conclude that $A(v)$ contains a monochromatic copy of $K_t$, which we denote by $B(v)$. Now, let $W$ be the set of all vertices $v\in V(\cG)$ in which $B(v)$ is blue. By symmetry between blue and red, we can assume that $|W| \geq \frac {1}{2} |V(\cG)|$. Let  $N=|W| \geq \frac {n}{2}$.
    
We define $B(W) = \bigcup_{v \in W} B(v)$ and take $\varphi$ to be the $(G[B(W)], \psi, s)$-colouring of $K_{N}$.  
If the blue subgraph of $K_N$ contains all trees in $\{T'|T\in \mathcal{T}_{c'n,d}\}$, then by Lemma~\ref{lem:chopping}, the blue subgraph of $G[B(W)]$ contains all graphs in $\mathcal{T}_{c'n,d}(k)$.
    	
From now on, we assume the blue subgraph of $K_N$ does not contain all trees in $\{T'|T\in \mathcal{T}_{c'n,d}\}$. Since each $T'$ in this family has $\Delta(T') \leq d^2$ and $N \geq 20c'nd^2(2k+1)$, by Lemma~\ref{aux:treeorpartite} there is a family of sets $V_0,V_1,\dots, V_{2k}\subseteq V(K_{N})$, each of size at least $\frac{N}{5d^2(2k+1)}$, such that for each $i\neq j$, the complete bipartite graph between $V_i$ and $V_j$ in $K_{N}$ contains only red edges.
   
Let an \emph{$i$-matching} in $\cG$ be a matching which consists of edges each incident to one vertex in $V_0$ and to one vertex in $V_i$, where $i\in[2k]$. In what follows, we construct a set $S \subseteq V_0$ of size $|S| \geq 2^{-{2k}}|V_0|$ and $2k$ many $i$-matchings $\{M_i\}_{i=1}^{2k}$, each of which covers $S$. This is done inductively on $i$, taking $S_0 := V_0$ as the base case with $i=0$. Suppose for some $j \leq 2k-1$, we have a set $S_j \subseteq V_0$ such that $|S_j| \geq 2^{-{j}}|V_0|$ and $j$ many $i$-matchings $\{M_i\}_{i=1}^{j}$ such that $M_i$ covers $S_j$ for each $i\in[j]$. Take a maximum matching $M_{j+1}$ between $S_j$ and $V_{j+1}$. Suppose for contradiction that $M_{j+1}$ has less than $|S_j|/2$ edges. Consider the vertex sets $X\subset S_j$ and $Y \subset V_{j+1}$ consisting of all vertices that are not incident to edges in $M_{j+1}$. Note that by the maximality of $M_{j+1}$, there are no edges between $X$ and $Y$. Since $|X|, |Y| \geq |S_j|/2 \geq 2^{-2k-2}|V_0| > \alpha n$, this contradicts $\cG$ being $\alpha$-joint. Therefore, at least $|S_j|/2 \geq |V_0|\cdot 2^{-(j+1)}$ vertices of $S_j$ are covered by $M_{j+1}$. Setting $S_{j+1}=V(M_{j+1})\cap S_j$ at each step, we get the set $S:=S_{2k}$ after $2k$ steps, which has the desired properties.
  
Let $v_i \in V_i$ be the only neighbour of $v$ in $M_i$, where $v\in S$ and $i \in [2k]$. Since $|S|\geq 2^{-2k}|V_0|> 20c' (d+1)n$, $\cG[S]$ contains all trees in $\mathcal{T}_{c'n,d}$ by Lemma~\ref{lem:joint expands} and Lemma~\ref{lem:friedman}, having in mind that $c'\gg \alpha$ which shows that $\cG$ is also $2c'$-joint. 
Let $T \boxtimes K_{k}$ be a member of $\mathcal{T}_{c'n,d}(k)$. Denote by $\tilde{T}$ the copy of $T$ as described which we can find in $\cG[S]$, and denote its vertex set by $U$. Pick a root $\tilde{r}$ of $\tilde{T}$ arbitrarily. 

For each $v \in V(\tilde{T})$, define $S(v) = \{v_1,v_2,\dots,v_k \}$ if the distance between $v$ and $\tilde{r}$ is even and $S(v) = \{v_{k+1},v_{k+2},\dots,v_{2k} \} $ if it is odd. Since the vertices in $S(v)$ all belong to different partition classes $V_i$, each $S(v)$ is a red clique in $K_N$, and note that it is also disjoint from all other red cliques $S(u)$ with $u \in V(\tilde{T})$. For every edge $uv \in E(\tilde{T})$, each edge of $K_N$ incident to a vertex $u'$ in $S(u)$ and another vertex $v'$ in $S(v)$ is red, since $u'$ and $v'$ cannot be in the same partition class $V_i$. Therefore, the graph induced by $\bigcup_{v \in U}{S(v)}$ in the red subgraph of $K_N$ contains a copy of $T \boxtimes K_{k}$. We now `transfer' this copy to the red subgraph of $G[B(W)]$ coloured according to $\psi$. Notice that each edge in this copy of $T \boxtimes K_{k}$ is also an edge of $\cG^3$, because every two vertices $v_i, v_j \in S(v)$ with $v \in V(\tilde{T})$ are at distance at most $2$ in $\cG$, and every two vertices $u_i \in S(u), v_j \in S(v)$ with $uv \in E(\tilde{T})$ are at distance at most $3$ in $\cG$.

By definition, for each $uv \in E(\cG^3)$ such that $\varphi(uv)$ is red in $K_N$, all edges between $B(u)$ and $B(v)$ are present in $G$, comprising a complete bipartite graph $G_{uv}$, and there is no blue copy of $K_{s,s}$ in $G_{uv}$. By
Lemma~\ref{thm:kovari-sos-turan}, the number of blue edges in $G_{uv}$ is at most $(s-1)^{s} t^{2-1/s} + st +1  \leq \frac{t^2}{16dk}$. Let $F:=T \boxtimes K_k$ and let $F' \subseteq G$ be the union of all the red edges in $G_{uv}$ for all $uv \in E(F)$. From Lemma~\ref{lem:lifting} it follows that $F'$ contains a red copy of $T \boxtimes K_k$. Note that our choice of $T \in \cT_{c'n,d}$ was arbitrary, so conditioned on the blue subgraph of $K_N$ not containing all trees in $\{T'|T \in \cT_{c'n,d}\}$, $G$ contains a red copy of every graph in $\cT_{c'n,d}(k)$.
\end{proof}

\end{document}